\documentclass[article,11pt]{amsart}

\usepackage[T1]{fontenc}
\usepackage[applemac]{inputenc}
\usepackage[francais,english]{babel}
\usepackage{amssymb}
\usepackage{amsmath}
\usepackage{latexsym}
\usepackage{marvosym}
\usepackage{amsfonts}
\usepackage{graphicx}
\usepackage{hyperref}
\usepackage{color}

\usepackage{enumerate}
\hypersetup{
backref=true, 
pagebackref=true,
hyperindex=true, 
colorlinks=true, 
breaklinks=true, 
urlcolor= blue, 
linkcolor= blue, 
bookmarks=true, 
bookmarksopen=true, 
pdfauthor={Romain AZAIS}, 
}

\newcommand{\prob}{\mathbf{P}}
\newcommand{\esp}{\mathbf{E}}
\newcommand{\ud}{\text{\normalfont d}}

\newcommand{\calB}{\mathcal{B}}
\newcommand{\R}{\mathbf{R}}

\renewcommand{\phi}{\varphi}

\newcommand{\fin}{
$\hfill
\mathbin{\vbox{\hrule\hbox{\vrule height1.5ex \kern.6em
\vrule height1.5ex}\hrule}}$}

\newtheorem{prop1}{Proposition}[section]
	
\newtheorem{lem1}[prop1]{Lemma}

\newtheorem{theo}[prop1]{Theorem}
\newtheorem{hyp}[prop1]{Assumption}
\newtheorem{hyps}[prop1]{Assumptions}
\newtheorem{rem}[prop1]{Remark}

\email{\vspace{-0.2cm}romain.azais@inria.fr}
\email{\vspace{-0.2cm}dufour@math.u-bordeaux1.fr}
\email{\vspace{-0.7cm}anne.petit@u-bordeaux2.fr}

\begin{document}
\title[Nonparametric estimation of the jump rate for NHMRP's]
{Nonparametric estimation of the jump rate for non-homogeneous marked renewal processes}
\author{Romain Aza\"{\i}s}
\author{François Dufour}
\author{Anne Gégout-Petit}
\address{
INRIA Bordeaux Sud-Ouest, team CQFD, France and Universit\'e Bordeaux, IMB, CNRS UMR 5251,
351, Cours de la Libération, 33405 Talence cedex, France.
}

\thanks{This work was supported by ARPEGE program of the French National Agency of Research (ANR), project “FAUTOCOES”, number ANR-09-SEGI-004.}

\begin{abstract}
This paper is devoted to the nonparametric estimation of the jump rate and the cumulative rate for a general class of non-homogeneous marked renewal processes, defined on a separable metric space. In our framework, the estimation needs only one observation of the process within a long time. Our approach is based on a generalization of the multiplicative intensity model, introduced by Aalen in the seventies. We provide consistent estimators of these two functions, under some assumptions related to the ergodicity of an embedded chain and the characteristics of the process. The paper is illustrated by a numerical example.\\
~

\noindent\textsc{Résumé.} Ce papier est consacré à l'estimation non-paramétrique du taux de saut et du taux de saut cumulé pour une classe générale de processus de renouvellement marqués non-homogènes, définis sur un espace métrique séparable. Dans notre cadre de travail, l'estimation nécessite seulement une observation du processus en temps long. Notre approche est basée sur une généralisation du modèle à intensité multiplicative introduit par Aalen dans les années soixante-dix. Nous donnons des estimateurs consistants de ces deux fonctions, sous des hypothèses portant sur l'ergodicité d'une chaîne immergée et sur les caractéristiques du processus. Le papier est illustré par un exemple numérique.
\end{abstract}
\maketitle

\section{Introduction}

The purpose of this paper is to develop a nonparametric method for estimating the jump rate of a non-homogeneous marked renewal process, when only one observation of the process within a long time is available. Our estimation procedure is premised on a generalization of the well-known multiplicative intensity model, investigated by Aalen in \cite{Aal77,Aal78}.

We introduce a general class of non-homogeneous marked renewal processes (NHMRP's), defined on an open subset of a separable metric space. The motion of the process depends on three characteristics namely the jump rate $\lambda$, which specifies the interarrival times, the transition kernel $Q$, and a function $t^\star$, which plays the role of a deterministic censorship depending on the state of the process. In this framework, the jump rate $\lambda$ is a function of two variables: a spatial mark and time. Here, our aim is to propose a nonparametric method for estimating both the jump rate $\lambda$ and the cumulative rate from only one observation of the process within a long time interval. In addition, the class of non-homogeneous marked renewal processes which we consider may be related to particular piecewise-deterministic Markov processes (PDMP's, see the book \cite{Dav}), whose transition kernel does not depend on time. This paper is also a keystone of \cite{Az}, in which we provide a consistent nonparametric estimator of the conditional density associated to the jump rate of a general PDMP.

Aalen suggested in the middle of the seventies the famous multiplicative intensity model (see his PhD thesis \cite{AalPHD}, or \cite{Aal77,Aal78}). In this work \cite{AalPHD,Aal77,Aal78}, it is assumed that the intensity of the underlying counting process $N$ can be written as the product of a predictable process $Y$ and a deterministic function $\lambda$, called jump rate or hazard rate. In this framework, the so-called Nelson-Aalen estimator provides an estimate of the cumulative rate $\Lambda(t)=\int_0^t\lambda(s)\ud s$. Ramlau-Hansen suggested a few years later, in \cite{Ram83}, a nonparametric method for estimating directly the rate $\lambda$, by smoothing the Nelson-Aalen estimator by kernel methods.

A large number of estimation problems in survival analysis or in statistics of processes are related to counting processes depending on a spatial variable. This one may be seen as a mark or as a covariate. In this context, Aalen's estimator is proved to be a very popular and powerful method, since the multiplicative assumption is satisfied in a large variety of applications (see for instance \cite{And}). In particular, one can apply the Nelson-Aalen approach for estimating the jump rate of a marked counting process, whose state space is finite, from a large number of independent observations. More recently in 2011, Comte \textit{et al.} proposed in \cite{COMTE} an adaptive method for estimating the jump rate of a marker-dependent counting process, under the multiplicative assumption.

There exists an extensive literature on nonparametric and semiparametric estimation methods when the spatial mark belongs to a continuous state space. However, we do not attempt to present an exhaustive survey on this topic. A significant list of references on this research field can be found in \cite{AalenHistory,And,MR676128,Martinussen2006} and the references therein. In particular, McKeague and Utikal were interested in \cite{McK90} in the estimation of the jump rate when the covariate takes its values in $[0,1]$. Their approach is based on smoothing a Nelson-Aalen type estimator both in spatial and time directions. In particular, they prove the consistency of their estimator. Li and Doss chose another approach, based on a local linear fit in the spatial direction (see \cite{MR1345201}). They extended McKeague and Utikal's work for the multidimensional case, and proved weak convergence results. One may also refer to the papers written by Utikal \cite{MR1208878,MR1445041} about jump rate estimation for two special classes of marked counting processes, under some continuous-time martingale assumptions. The Euclidean structure of the covariate state space is a keystone of the papers mentioned above. At the same time, nonparametric approaches have also been considered by Beran in \cite{BERAN}, Stute in \cite{MR840519} and Dabrowska in \cite{MR932943}, but for independent observations. Semiparametric methods have also been considered by many authors, beginning with Cox in \cite{Cox72}. The interested reader may consult the book \cite{And} and the references therein for a complete review of the literature on these models.

In many aspects, our approach and the results mentioned above are different and complementary. These differences may be briefly described as follows. Our paper is based on a generalization of the multiplicative intensity model, involving a discretization of the state space, and, as a consequence, an approximation of the functions of interest. Indeed, we do not impose any conditions on the state space, such as to be Euclidean. This notably excludes the methods investigated by McKeague and Utikal \cite{McK90}, Li and Doss \cite{MR1345201}, or Utikal \cite{MR1208878,MR1445041}. Furthermore, these authors consider some assumptions about both continuous-time martingale properties and the asymptotic behavior of $Y$. From a practical point of view, this kind of assumption is not completely satisfactory due to the fact that these conditions may be difficult to check, especially in our case. On our side, we overcome this difficulty by providing tractable conditions, directly related to the primitive data of the process, to ensure the consistency of our estimator.

The present paper is divided into two parts. In the first one, we consider that the transition kernel only charges a finite set of points. It amounts to considering the state space is a discrete one. In this context, Theorem \ref{th:mgdiscret} states that the multiplicative model is satisfied. In the second part, we assume that the kernel $Q$ is diffuse, that is to say, it does not charge singletons. If the $Z_i$'s denote the marks of the underlying process, for any $x$ and $i$, the indicator function $\mathbf{1}_{\{Z_i=x\}}$ is almost surely null. This rules out the method developed in the discrete case. Our procedure relies on a partition of the state space, labeled $(A_k)$. In this context, it appears intuitive to consider the counting process
$$N_n(A_k,t)=\sum_{i=0}^{n-1} \mathbf{1}_{\{Z_i\in A_k\}}\mathbf{1}_{\{S_{i+1}\leq t\}},$$
where the $S_i$'s denote the interarrival times. Although Aalen's multiplicative model does not hold for this counting process, the stochastic intensity of $N_n(A_k,t)$ is almost surely equivalent to the product $Y_n(A_k,t)l(A_k,t)$ (see Proposition \ref{defl}), where $l(A_k,t)$ is an approximation of the jump rate $\lambda(x,t)$, for $x\in A_k$, and
$$Y_n(A_k,t)=\sum_{i=0}^{n-1}\mathbf{1}_{\{Z_i\in A_k\}}\mathbf{1}_{\{S_{i+1}\geq t\}}.$$
In this context, it is natural to introduce the following processes,
$$\widehat{L}_n(A_k,t) = \int_0^tY_n(A_k,s)^+\ud N_n(A_k,s) ,$$
where $Y_n(A_k,t)^+$ is the generalized inverse of $Y_n(A_k,t)$, and
$$L_n^\ast(A_k,t) = \int_0^t l(A_k,s)\mathbf{1}_{\{Y_n(A_k,s)>0\}}\ud s,$$
In Aalen's papers, the difference $\widehat{L}_n(A_k,t)-L_n^\ast(A_k,t)$ is a continuous-time martingale, whereas on our part, it is not the case since there exists an extra-term $a_n(t)$, which vanishes when $n$ goes to infinity. Intuitively, this means that the multiplicative model asymptotically makes sense. Referring to both Lenglart's inequality and the asymptotic behavior of the extra-term $a_n(t)$, we prove that $\widehat{L}_n(A_k,t)$ is a consistent estimator of $L(A_k,t)=\int_0^tl(A_k,s)\ud s$ (see Proposition \ref{pppaaa}). We deduce from this a consistent estimator of the cumulative rate $\Lambda(x,t)=\int_0^t\lambda(x,s)\ud s$ (see Theorem \ref{est-sim-final}), since $L(A_k,t)$ and $\Lambda(x,t)$ are close for $x\in A_k$ (see Lemma \ref{txlip}). Next, we focus on smoothing this estimator by kernel methods in order to suggest a consistent estimator of $l(A_k,t)$ (see Proposition \ref{kernel}) and, therefore, of $\lambda(x,t)$ (see Theorem \ref{pms:theo:CVfin}). An inherent difficulty throughout this paper is related to the presence of the deterministic censorship.

The paper is organized in the following way. We first give, in Section \ref{s:def}, the precise definition of the class of non-homogeneous marked renewal processes which we are interested in, and we provide an example of application in reliability. We state also some technical results about continuous-time martingales and conditional independences. Section \ref{s:discrete} is devoted to the discrete case, where we consider that the transition kernel $Q$ only charges a finite number of points. The main contribution of the paper lies in Section \ref{s:continuous}, in which we do not impose any conditions on the state space. In this part, we provide consistent estimators of the cumulative rate (see Theorem \ref{est-sim-final}) and the jump rate (see Theorem \ref{pms:theo:CVfin}). Finally in Section \ref{s:simu}, we present a numerical example to illustrate the good behavior of our estimators on finite sample size.

		\section{Definition and first results}
		\label{s:def}
		
In this section, we first define the class of non-homogeneous marked renewal processes under consideration. We consider a piecewise-constant continuous-time process $(X_t)_{t\geq0}$. For any $t\geq0$, $X_t$ takes its values on $E$, which is an open subset of a separable metric space $(\mathcal{E},d)$. The motion of $(X_t)_{t\geq0}$ may be given for any $t\geq0$ in the following way,
$$\forall t\geq0,~X_t=Z_n \quad\text{if}\quad S_0+\dots + S_n\leq t<S_0+\dots+S_{n+1}.$$
The $Z_n$'s correspond to the locations of $(X_t)_{t\geq0}$, while the $S_n$'s denote the interarrival times. We assume that $(Z_n)_{n\geq0}$ is a Markov chain on $(E,\calB(E))$, defined on a probability space $(\Omega,\mathcal{A},\prob_{\nu_0})$, whose transition kernel is denoted by $Q$. The distribution of the starting point $Z_0$ is assumed to be $\nu_0$. An equivalent formulation (see for instance Theorem 2.4.3 of \cite{BEK}) is given by,
	\begin{equation}
		\label{dyn:Z}
		\forall n\geq1,~Z_n=\psi(Z_{n-1},\varepsilon_{n-1}),
	\end{equation}
where $\psi$ is a measurable function, and $(\varepsilon_n)_{n\geq0}$ is a sequence of independent and identically distributed random variables. Now, one defines the sequence $(S_n)_{n\geq0}$ on $(\mathbf{R}_+,\calB(\R_+))$ from two functions: $\lambda:E\times\mathbf{R}_+\to\R_+$ which plays the role of the jump rate, and the deterministic censorship $t^\star:E\to ]0,+\infty]$. For each integer $n\geq1$, the distribution of $S_n$ satisfies, for any $t\geq0$,
\begin{eqnarray}
\quad \prob_{\nu_0}(S_n>t|\{Z_i : i\geq0\},S_0,\dots,S_{n-1}) &=& \prob_{\nu_0}(S_n>t|Z_{n-1}) \label{dyn:S} \\
&=& \exp\left(-\int_0^t\lambda(Z_{n-1},s)\ud s\right)\mathbf{1}_{\{0\leq t<t^\star(Z_{n-1})\}} . \nonumber
\end{eqnarray}
In addition, we assume that $S_0=0$. Hence, there exist a function $\phi$ and a sequence of independent and identically distributed random variables $(\delta_n)_{n\geq0}$, which is independent of the sequence $(\varepsilon_n)_{n\geq0}$, such that,
$$\forall n\geq1,~S_n=\phi(Z_{n-1},\delta_{n-1}).$$
One assumes that both the sequences $(\varepsilon_n)_{n\geq0}$ and $(\delta_n)_{n\geq0}$ are independent of $Z_0$.

We recall that we are interested here in the nonparametric estimation of the jump rate $\lambda$, from one observation of the embedded chain $(Z_n,S_n)_{n\geq0}$. This class of non-homogeneous renewal models may be related to PDMP's, for which the transition kernel $Q$ does not depend on time. Hence, the estimation method developed in this paper is very useful in order to estimate the conditional distribution of the interarrival times for PDMP's (see \cite{Az}). Nevertheless, providing a method for estimating the jump rate for this class of stochastic models has an intrinsic interest. In the following, we present an example in reliability of non-homogeneous marked renewal process satisfying the model mentioned above.

Let us consider a machine, whose production configuration takes its values in an open subset of $\mathbf{R}^d$. The dynamic of the regime is assumed to be a non-homogeneous renewal process: the state is piecewise-constant until a failure spontaneously occurs. One naturally considers that the failure rate depends on the production regime. When the machine breaks down, the repair occurs instantaneously, and the machine configuration is randomly changed, according to a transition kernel $Q$ depending only on the previous working state. In addition, one may consider that there exists a deterministic period of inspection, depending on the production configuration too. The inspection is instantaneous and the next regime changes according to the kernel $Q$. The estimation of the failure rate from only one observation of the regime state within a long time may bring some informations about the behavior of the production machine. The main benefit of this approach is as follows: the estimation does not need the observation of a great number of similar machines.

One may associate to the jump rate $\lambda$, the cumulative rate, the survival function and the probability density function, which are related to it. The conditional density $f$ satisfies,
	\begin{equation}
			\label{pms:def:f}
		\forall \xi\in E, ~ \forall t\geq0,~f(\xi,t)  = \lambda(\xi,t) \exp\left( - \int_0^t {\lambda}( \xi,s)\ud s\right).
	\end{equation}
The cumulative rate $\Lambda$ is defined by,
	\begin{equation}
			\label{pms:def:LAMBDA}
		\forall \xi\in E,~\forall t\geq0,~\Lambda(\xi,t) = \int_0^t {\lambda}( \xi,s) \ud s .
	\end{equation}
Finally, the conditional survival function is denoted by $G$.
	\begin{equation}
			\label{pms:def:G}
		\forall \xi \in E,~\forall t\geq0,~ G(\xi,t)= \exp\left(-\int_0^t \lambda(\xi,s)\ud s\right).
	\end{equation}
We have the straightforward relation between these functions: $\lambda=f/G$. For each $n\geq1$, $\mathcal{G}_n$ denotes the $\sigma$-field generated by the $n$ first $Z_i$'s,
	\begin{equation}
		\label{def:Gn}	
			\mathcal{G}_n = \sigma(Z_0 , \dots ,Z_{n-1}) .
	\end{equation}
For each integer $i$, the one-jump counting process $N^{i+1}$ is given for any $t\geq0$, by
$$N^{i+1}(t) = \mathbf{1}_{\{S_{i+1}\leq t\}},$$
and $(\mathcal{F}_t^{i+1})_{t\geq0}$ denotes the associated filtration. In this section, we shall prove two results: the first one is related to two conditional independence properties; the second one deals with the continuous-time martingale associated to the counting process $N^{i+1}$ in a special filtration.

		\begin{prop1}	\label{prop:indcond}
		Let $n$ be an integer and $1\leq i\leq n$. For each integer $j\neq i$, let $t_j\geq0$ and $t\geq0$. Then, we have
		$$ \bigvee_{j\neq i} \mathcal{F}_{t_j}^{j}  ~   \underset{\mathcal{G}_n}{\bot} ~ \mathcal{F}_t^i 	\qquad\text{and}\qquad    \mathcal{F}_t^i ~ \underset{\sigma(Z_{i-1})}{\bot} ~ \mathcal{G}_n .$$
		Furthermore, we deduce from Proposition 6.8 of \cite{Kal} this immediate corollary: for any $s<t$, $0\leq i\leq n-1$,
		$$
		\bigvee_{j\neq i+1} \mathcal{F}_{s}^j  ~   \underset{\mathcal{G}_n \vee \mathcal{F}_s^{i+1}}{\bot} ~ \mathcal{F}_t^{i+1}
		\qquad\text{and}\qquad
		\mathcal{F}_t^{i+1} ~\underset{\sigma(Z_i) \vee \mathcal{F}_s^{i+1}}{\bot}~\mathcal{G}_n .
		$$
		\end{prop1}
		
		\begin{proof}
		The reader may find the proof in Appendix \ref{app1}.
		\end{proof}

\noindent
We have also a continuous-time martingale property. This is the one associated to the counting process $N^{i+1}$.

		\begin{lem1}																	
			\label{lem:mg}
		For each integer $i$, the process $M^{i+1}$ given by,
		\begin{equation}\label{Miplus1}
			\forall 0\leq t<t^\star(Z_i),~M^{i+1} (t) = N^{i+1}(t) - \int_0^t \lambda(Z_i , u) \mathbf{1}_{\{S_{i+1} \geq u \}} \ud u ,
		\end{equation}
		is a continuous-time martingale in the filtration $(\sigma(Z_i)\vee\mathcal{F}_s^{i+1})_{0\leq s<t^\star(Z_i)}$.
		\end{lem1}
		
		\begin{proof}
		The proof is deferred in Appendix \ref{app1}.
		\end{proof}

Proposition \ref{prop:indcond} and Lemma \ref{lem:mg} are prominent for the next results. In addition, the reference to the underlying probability measure $\prob_{\nu_0}$ will be implicit in the text. For the sake of readability, we shall write $\prob_x$ instead of $\prob_{\delta_{\{x\}}}$.

		\section{Discrete state space}
		\label{s:discrete}
		
		We assume here that the transition kernel $Q$ only charges a finite set which we denote $\{x_1,\dots,x_M\}$. One may associate to each $x_i$ the deterministic exit time $t_i^\star=t^\star(x_i)$. Now, one considers that $k$ is fixed. In this section, we shall prove in Theorem \ref{th:mgdiscret} that the multiplicative model is satisfied for estimating the cumulative rate $\Lambda(x_k,\cdot)$.
		
		For each integer $n$, let us introduce the counting process $N_n(x_k,\cdot)$ by,
		\begin{equation}	\label{pms:def:Ndiscret}
			\forall t \geq 0, ~ N_n (x_k , t) = \sum_{i=0}^{n-1} \mathbf{1}_{\{S_{i+1}\leq t\}} \mathbf{1}_{\{Z_i = x_k\}} .
		\end{equation}
		In addition, for any $t\geq0$, we define
		\begin{equation}	\label{pms:def:Ydiscret}
			Y_n(x_k , t) = \sum_{i=0}^{n-1} \mathbf{1}_{\{S_{i+1} \geq t\}} \mathbf{1}_{\{ Z_i = x_k \}} .
		\end{equation}

		\begin{theo}										
		\label{th:mgdiscret}
		Let $n\geq1$. The process $M_n(x_k , \cdot)$ defined by,
		\begin{equation} \label{pms:def:Mdiscret}
		 \forall 0\leq t < t_k^\star, ~ M_n(x_k , t) = N_n(x_k , t) - \int_0^t \lambda(x_k , s) Y_n(x_k,s) \ud s
		 \end{equation}
		is a $(\widetilde{\mathcal{F}}_t^n)_{0\leq t<t_k^\star}$-continuous-time martingale under $\prob_{\nu_0}$, with,
		$$ \forall 0 \leq t < t_k^\star, ~ \widetilde{\mathcal{F}}_t^n = \mathcal{G}_n \vee \bigvee_{i=0}^{n-1} \mathcal{F}_t^{i+1}.$$
		\end{theo}
		
\begin{proof}
Let $0\leq s<t<t^\star_k$. Plugging $(\ref{pms:def:Ndiscret})$ and $(\ref{pms:def:Ydiscret})$ in $(\ref{pms:def:Mdiscret})$, we have
\begin{eqnarray}
M_n(x_k , t) &=& \sum_{i=0}^{n-1} \Big\{ N^{i+1}(t) \mathbf{1}_{\{Z_i = x_k\}}  -  \int_0^t \lambda(x_k,u) \mathbf{1}_{\{S_{i+1}\geq u\}} \mathbf{1}_{\{Z_i = x_k\}} \ud u \Big\} \nonumber \\
&=& \sum_{i=0}^{n-1} M^{i+1}(t) \mathbf{1}_{\{Z_i=x_k\}},\label{sumMG}
\end{eqnarray}
by $(\ref{Miplus1})$. Thus, we obtain
$$\esp_{\nu_0}[ M_n(x_k , t) |\widetilde{\mathcal{F}}_s^n ]  = \sum_{i=0}^{n-1}   \Big\{   \esp_{\nu_0}[ N^{i+1}(t) | \widetilde{\mathcal{F}}_s^n ] - \esp_{\nu_0}[ \int_0^t  \mathbf{1}_{\{S_{i+1} \geq u\}} \lambda(x_k , u) \ud u | \widetilde{\mathcal{F}}_s^n ] \Big\} \mathbf{1}_{\{ Z_i = x_k \}}  .$$
On the strength of Proposition \ref{prop:indcond}, we have
$$ \bigvee_{j\neq i+1} \mathcal{F}_{s}^j  ~   \underset{\mathcal{G}_n \vee \mathcal{F}_s^{i+1}}{\bot} ~ \mathcal{F}_t^{i+1}  .$$
Together with Proposition 6.6 of \cite{Kal}, we deduce
\begin{align*}
&\esp_{\nu_0}[ M_n(x_k , t) |\widetilde{\mathcal{F}}_s^n ] \\
&~= \sum_{i=0}^{n-1}   \Big\{   \esp_{\nu_0}[ N^{i+1}(t) | \mathcal{G}_n \vee \mathcal{F}_s^{i+1} ]  - \esp_{\nu_0}[ \int_0^t  \mathbf{1}_{\{S_{i+1} \geq u\}} \lambda(x_k , u) \ud u |\mathcal{G}_n \vee \mathcal{F}_s^{i+1} ]   \Big\}  \mathbf{1}_{\{ Z_i = x_k \}} .
\end{align*}
Furthermore, from Proposition \ref{prop:indcond}, we have
$$ \mathcal{F}_t^{i+1} ~\underset{\sigma(Z_i) \vee \mathcal{F}_s^{i+1}}{\bot}~\mathcal{G}_n .$$
Thus, in the light of Proposition 6.6 of \cite{Kal} again, we have
\begin{eqnarray*}
\esp_{\nu_0}[ M_n(x_k , t) |\widetilde{\mathcal{F}}_s^n ] &=&
 \sum_{i=0}^{n-1}   \Big\{   \esp_{\nu_0}[ N^{i+1}(t) | \sigma(Z_i) \vee \mathcal{F}_s^{i+1} ] \mathbf{1}_{\{ Z_i = x_k \}}  \\
 ~&~& -~ \esp_{\nu_0}[ \int_0^t  \mathbf{1}_{\{S_{i+1} \geq u\}} \lambda(x_k , u) \ud u | \sigma(Z_i) \vee \mathcal{F}_s^{i+1} ] \mathbf{1}_{\{Z_i = x_k\}}   \Big\} .
 \end{eqnarray*}
 By Lemma \ref{lem:mg}, this yields to
$$ \esp_{\nu_0}[ M_n(x_k , t) |\widetilde{\mathcal{F}}_s^n ] = \sum_{i=0}^{n-1} M^{i+1}(s) \mathbf{1}_{\{Z_i = x_k\}} .$$
Finally, together with $(\ref{sumMG})$, $M_n(x_k,\cdot)$ is, therefore, a martingale.
\end{proof}

\noindent
Theorem \ref{th:mgdiscret} states that one may estimate the cumulative rate $\Lambda(x_k,t)$ with the Nelson-Aalen estimator $\widehat{\Lambda}_n(x_k , t)$ given by
$$ \widehat{\Lambda}_n(x_k , t) = \sum_{i=0}^{n-1} \frac{1}{Y_n(x_k , S_{i+1})} \mathbf{1}_{\{S_{i+1} \leq t\}} \mathbf{1}_{\{Z_i=x_k\}} .$$
One refers the interested reader to \cite{Aal77,Aal78} or \cite{And} for a general survey on the properties of this estimator.

		\section{Continuous state space} \label{s:continuous}
			
			The present section is divided into two parts. In the first one, we provide an estimate of the cumulative rate $\Lambda$. The second part deals with the estimation of the jump rate $\lambda$ by smoothing the estimator of $\Lambda$ by kernel methods.

		\subsection{Estimation of $\Lambda$}
		
Let us assume that the transition kernel $Q$ is diffuse, that is to say, $Q$ does not charge singletons. The previous procedure is ruled out, since for any $x\in E$ and each integer $i$, $\mathbf{1}_{\{Z_i=x\}} = 0$ almost surely. As a consequence, we shall naturally approximate under regularity conditions the jump rate in $x$, by the jump rate given the state is in a neighborhood of $x$. As mentioned in the introduction, we shall consider the counting process $N_n(A,t)$ defined for $A\in\calB(E)$ by
$$N_n(A,t)=\sum_{i=0}^{n-1}\mathbf{1}_{\{Z_i\in A\}}\mathbf{1}_{\{S_{i+1}\leq t\}} .$$
We will see $(\ref{pms:expr:Mn})$ that the stochastic intensity of $N_n(A,t)$ in a well-chosen filtration is
$$\sum_{i=0}^{n-1}\mathbf{1}_{\{Z_i\in A\}}\mathbf{1}_{\{S_{i+1}\geq t\}}\lambda(Z_i,t) .$$
The first results of this section deal with the ergodicity of the underlying Markov chains. Their properties are prominent to establish the asymptotic behavior of the stochastic intensity of $N_n(A,t)$ (see Proposition \ref{defl}). Based on these results, we will study in Proposition \ref{pppaaa} and Theorem \ref{cor-cons} the estimation of $L(A,t)=\int_0^tl(A,s)\ud s$, where $l(A,t)$ is an approximation of the jump rate $\lambda(x,t)$, with $x$ in $A$. We will deduce from this an estimator of the cumulative rate $\Lambda(x,t)$ (see Theorem \ref{est-sim-final}). An additional difficulty is related to the invariant measure of $A$, which we have to estimate.

We shall impose some assumptions about both the ergodicity of the Markov chain $(Z_n)_{n\geq0}$ and the characteristics of the process. In the following, $\nu_n$ denotes the distribution of $Z_n$, for each integer $n$.
		
		\begin{hyp}	\label{hyp:ergodic}
		There exists a probability measure $\nu$ such that, for any initial distribution $\nu_0=\delta_{\{x\}}$, $x\in E$,
		$$\lim_{n\to+\infty} \| \nu_n - \nu \|_{TV} = 0 .$$
		\end{hyp}
		
		\noindent
		This assumption may be directly related to the transition kernel of the Markov chain $(Z_n)_{n\geq0}$ (existence of a Foster-Lyapunov's function or Doeblin's condition for instance). We refer the interested reader to \cite{MandT} for results about this kind of connection. This assumption leads to the following results.
		
		\begin{prop1} \label{prop:ergodic}
		We have the following statements:\begin{enumerate}
		\item $(Z_n)_{n\geq0}$ is $\nu$-irreducible.
		\item $(Z_n)_{n\geq0}$ is positive Harris-recurrent and aperiodic.
		\item $\nu$ is the unique invariant probability measure of $(Z_n)_{n\geq0}$.
		\end{enumerate}
		\end{prop1}
		
		\begin{proof}
		The proof may be found in Appendix \ref{app2}.
		\end{proof}

Denote by $\eta_n$ the distribution of the couple $(Z_n,S_{n+1})$, and by $\mu_z(\cdot)$ the conditional law of $S_1$ given $Z_0=z$. By construction $(\ref{dyn:S})$, the distribution of $S_{n+1}$ given $Z_n=z$ is also given by $\mu_z(\cdot)$. The following result gives us the probability measure of $(Z_n,S_{n+1})$ according to $(\mu_z)_{z\in E}$ and $\nu_n$.

		\begin{lem1}	\label{desint}
		For each integer $n$, $\eta_n$ satisfies, for any $A\times\Gamma\in\mathcal{B}({E})\otimes\mathcal{B}(\mathbf{R}_+)$,
		$$ \eta_n(A\times\Gamma) = \int_{A\times\Gamma} \mu_z(\ud s) \nu_n(\ud z) .$$
		\end{lem1}
		
\begin{proof}
This is the disintegration of the measure $\eta_n$ according to its marginal distribution $\nu_n$.
\end{proof}

\noindent
We shall see that the sequence $(\eta_n)_{n\geq0}$ converges to the probability measure $\eta$ given by,
$$\forall A\times\Gamma \in \mathcal{B}({E})\otimes\mathcal{B}(\mathbf{R}_+),~\eta(A\times\Gamma) = \int_{A\times\Gamma} \mu_z(\ud s)\nu(\ud z) .$$

		\begin{lem1}	\label{lem37}
		For any initial distribution $\nu_0=\delta_{\{x\}}$, $x\in E$,
			$$ \lim_{n\to+\infty} \| \eta_n - \eta \|_{TV}  = 0 .$$
		\end{lem1}

		\begin{proof}
		The proof is given in Appendix \ref{app2}.
		\end{proof}

\noindent
As a consequence, the Markov chain $(Z_n,S_{n+1})_{n\geq0}$ has similar properties of the Markov chain $(Z_n)_{n\geq0}$ given in Proposition \ref{prop:ergodic}.

		\begin{prop1}\label{3point8}
		We have the following statements:													
		\begin{enumerate}
		\item $(Z_n,S_{n+1})_{n\geq0}$ is $\eta$-irreducible.
		\item $(Z_n,S_{n+1})_{n\geq0}$ is positive Harris-recurrent and aperiodic.
		\item $\eta$ is the unique (up to a multiple constant) invariant probability measure of the Markov chain $(Z_n,S_{n+1})_{n\geq0}$.
		\end{enumerate}
		\end{prop1}
		
\begin{proof}
One may state this result from Lemma \ref{lem37}, with the arguments given in the proof of Proposition \ref{prop:ergodic}.
\end{proof}

\noindent		
According to the previous discussion, we shall apply the ergodic theorem to the Markov chains $(Z_n)_{n\geq0}$ and $(Z_n,S_{n+1})_{n\geq0}$. Now, we impose some assumptions on the characteristics of the process.

		\begin{hyps} ~
		\begin{enumerate}
		\item The jump rate $\lambda$ is uniformly Lipschitz, that is,
		$$\exists [\lambda]_{Lip}>0, ~ \forall \xi,\xi'\in E, ~ \forall s \geq 0, ~ | \lambda (\xi , s) - \lambda(\xi',s)| \leq [\lambda]_{Lip} ~\! d(\xi , \xi') .$$
		\item There exists a locally integrable function $M:\mathbf{R}_+\to\mathbf{R}_+$ such that,
		$$\forall \xi\in E,~\forall s\geq0,~\lambda(\xi,s)\leq M(s) .$$
		\item The density $f$ is continuous in time.
		\item $f$ is bounded.
		\item The function $t^\star$ is continuous.
		\end{enumerate}
		\end{hyps}

\noindent
Under these assumptions, one states some intermediate results about $\lambda$, $t^\star$ and $G$.

\noindent

		\begin{lem1} \label{r001}
		Let $A\in\mathcal{B}(E)$ be a relatively compact set such that $\overline{A}\cap\partial E=\emptyset$. Thus,
		$$\inf_{\xi\in A} t^\star(\xi)>0 .$$
		In this case, one denotes $\displaystyle t^\star(A) = \inf_{\xi\in A} t^\star(\xi)$. Furthermore,
		$$\forall t\geq0,~\inf_{\xi\in A} G(\xi,t)>0 .$$
		\end{lem1}		
		\begin{proof}
		The proof is deferred in Appendix \ref{app2}.
		\end{proof}

Let us introduce the following notation.
$$
\mathcal{B}_\nu^+ = \big\{ A\in\mathcal{B}(E) ~:~A~\!\text{relatively compact},~\! \nu(A)>0~\!\text{and}~\! \overline{A}\cap\partial E = \emptyset \big\} .
$$
The following lemma states that, for any $A\in\mathcal{B}_\nu^+$, $\lambda$ is bounded on $A\times[0,t^\star(A)[$.

		\begin{lem1}
		Let $A\in\mathcal{B}_\nu^+$, $\xi\in A$ and $0\leq t<t^\star(A)$. Thus,					\label{majorlambda}
		$$ \lambda(\xi,t) \leq \frac{\|f\|_{\infty}}{ \inf_{\xi\in A} G\big(\xi,t^\star(A)\big)} .$$
		\end{lem1}

	\begin{proof}
	As $f$ is bounded and $G(\xi,\cdot)$ is decreasing, we obtain by Lemma \ref{r001},
	\begin{eqnarray*}
	\lambda(\xi,t) &\leq& \frac{ \|f\|_{\infty} }{ G(\xi,t) }  \leq  \frac{ \|f\|_{\infty} }{ G\big(\xi,t^\star(A)\big) }.
	\end{eqnarray*}
	This immediately yields to the expected result.
	\end{proof}

\noindent
Let $A\in\mathcal{B}_{\nu}^+$. Let us consider for each integer $n$ the continuous-time process $Y_{n}(A, \cdot)$, defined by,
\begin{equation}
\label{eq:Y} \forall 0\leq t<t^\star(A), ~Y_{n}( A  , t) 	=	\sum_{i=0}^{n-1} \mathbf{1}_{\{S_{i+1} \geq t\}} \mathbf{1}_{\{Z_i \in A \}}  	.
\end{equation}
We shall state two lemmas about the asymptotic behavior of $Y_n$. Before, we focus our attention on the link between $G(z,\cdot)$ and $\mu_z(\cdot)$, for any $z\in A$.

\begin{rem}															\label{rem:survie}
Let $A\in\mathcal{B}_{\nu}^+$ and $z\in A$. For any $0\leq t<t^\star(A)$, we have
\begin{eqnarray*}
\mu_z\big([t,+\infty[\big)  &=& \prob_{\nu_0} (S_1\geq t |Z_0 = z) \\
~&=& \prob_{\nu_0}(S_1>t |Z_0=z) ,
\end{eqnarray*}
because $t<t^\star(A)\leq t^\star(z)$. Thus,
$$\mu_z\big([t,+\infty[\big) = G(z,t) .$$
\end{rem}

		\begin{lem1} \label{lem:ynsurn}
		Let $A\in\mathcal{B}_\nu^+$ and $x\in E$. Thus, for any $0\leq t<t^\star(A)$,
		$$\frac{Y_n(A,t)}{n} \longrightarrow \int_{A} G(z,t)\nu(\ud z) \quad \text{$\prob_{x}$-a.s. as $n\to+\infty$}.$$
		Furthermore, this limit is strictly positive.
		\end{lem1}

\begin{proof}
In the light of Proposition \ref{3point8}, the Markov chain $(Z_n,S_{n+1})_{n\geq0}$ is positive Harris-recurrent and admits $\eta$ as its unique invariant probability measure. Thus, using the ergodic theorem (see for instance Theorem 17.1.7 of \cite{MandT}), we have
$$ \frac{Y_n(A,t)}{n} \to  \int_{A} \mu_z\big([t,+\infty[\big) \nu(\ud z) \quad \text{$\prob_{x}$-\textit{a.s.} as $n\to+\infty$}.$$
Furthermore, for any $z\in A$, as $t<t^\star(A)$ and according to Remark \ref{rem:survie},
$$\mu_z\big([t,+\infty[\big) = G(z,t) .$$
It is a strictly positive number because $\nu(A)>0$ and $\inf_{\xi\in A} G(\xi,t) >0$ by Lemma \ref{r001}.
\end{proof}

\noindent
Let $A\in\mathcal{B}_{\nu}^+$ and $0\leq t<t^\star(A)$. Let us introduce the generalized inverse $Y_n(A,t)^+$ of $Y_n(A,t)$ by
\begin{equation}	\label{invY}
Y_n(A,t)^+ = 
\left\{
\begin{array}{cll}
0 &\text{if}&Y_n(A,t) =0 ,\\
\frac{1}{Y_n(A,t)} &\text{else.}&
\end{array}
\right.
\end{equation}

		\begin{lem1}																				\label{S003}
		Let $A\in\mathcal{B}_\nu^+$, $0\leq t<t^\star(A)$ and $x\in E$. Thus, for all integers $n$,
		\begin{equation}	\label{ynplus}
		Y_n(A,t)^+\leq 1~\textit{$\prob_x$-a.s.}
		\end{equation}
		and, as $n$ goes to infinity,
		\begin{eqnarray}
		Y_n(A,t)^+ 			&\longrightarrow&0\quad \prob_x\textit{-a.s.},						\label{limyn+} 	\\
		\mathbf{1}_{\{Y_n(A,t)=0\}} 	&\longrightarrow & 0\quad\prob_x\textit{-a.s.},					\label{lim1yn} \\
		\int_0^t \mathbf{1}_{\{Y_{n}(A , s) = 0  \}}  \ud s &\longrightarrow&0 \quad\prob_x\textit{-a.s.}		\nonumber
		\end{eqnarray}
		\end{lem1}

\begin{proof}$Y_n(A,t)^+$ is almost surely bounded by $1$, since $Y_n(A,t)$ takes its values on the integers. One immediately obtains the limits $(\ref{limyn+})$ and $(\ref{lim1yn})$, because $Y_n(A,t)/n$ almost surely admits a strictly positive limit by virtue of Lemma \ref{lem:ynsurn}. Finally,
$$
\limsup_{n\to+\infty} \int_0^t \mathbf{1}_{\{Y_{n}(A , s) = 0  \}}  \ud s  \leq\int_0^t \limsup_{n\to+\infty} \mathbf{1}_{\{Y_{n}(A , s) = 0  \}}  \ud s ~=~0,
$$
by $(\ref{lim1yn})$.
\end{proof} 

\noindent
In the following proposition, we shall apply the ergodic theorem in order to define, for any $A\in\mathcal{B}_\nu^+$, the function $l(A,\cdot)$ which is an approximation of the jump rate $\lambda(\xi,\cdot)$, for $\xi\in A$ (see Lemma \ref{txlip}). We also state the continuity of this function.

		\begin{prop1} \label{defl}
		Let $A\in\mathcal{B}_{\nu}^+$, $0\leq t<t^\star(A)$ and $x\in E$. Thus, when $n$ goes to infinity,
		\begin{equation}	\label{def:petitl}
		Y_n(A,t)^+~\! \sum_{i=0}^{n-1} \lambda(Z_i,t) \mathbf{1}_{\{Z_i\in A\}} \mathbf{1}_{\{S_{i+1}\geq t\}} \longrightarrow  l(A,t) = \frac{ \int_{A} f(z,t) \nu(\ud z)}{ \int_{A} G(z,t) \nu(\ud z)}~\textit{$\prob_{x}$-a.s.}
		 \end{equation}
		 The function $l(A,\cdot)$ is continuous on $[0,t^\star(A)[$. We especially have
		 \begin{equation} \label{Kt}
		 K_t(A)=\sup_{0\leq s\leq t} \big| l(A,s) \big| < +\infty .
		 \end{equation}
		\end{prop1}

\begin{proof}
The ergodic theorem (Theorem 17.1.7 of \cite{MandT}) applied to $(Z_n,S_{n+1})_{n\geq0}$ leads to
\begin{eqnarray*}
\lim_{n\to+\infty}  \frac{1}{n} \sum_{i=0}^{n-1} \lambda(Z_i,t) \mathbf{1}_{\{Z_i\in A\}} \mathbf{1}_{\{S_{i+1}\geq t\}}   &=&   \int_A \lambda(z,t) \mu_z\big([t,+\infty[\big) \nu(\ud z) ~\textit{$\prob_{x}$-a.s.} \\
~&=& \int_A \lambda(z,t) G(z,t) \nu(\ud z) ~\textit{$\prob_{x}$-a.s.},
\end{eqnarray*}
because $t<t^\star(A)$. Notice that $f(z,t) = \lambda(z,t)G(z,t)$. Thus,
$$
\lim_{n\to+\infty}  \frac{1}{n} \sum_{i=0}^{n-1} \lambda(Z_i,t) \mathbf{1}_{\{Z_i\in A\}} \mathbf{1}_{\{S_{i+1}\geq t\}}   =    \int_A f(z,t) \nu(\ud z) ~\textit{$\prob_{x}$-a.s.}
$$
Furthermore, in the light of Lemma \ref{lem:ynsurn},
\begin{eqnarray*}
\lim_{n\to+\infty}  \frac{1}{n} Y_n(A,t) = \int_A G(z,t)\nu(\ud z) ~\textit{$\prob_{x}$-a.s.}
\end{eqnarray*}
Finally, doing the ratio of these two limits leads to the expected convergence. On the strength of Lebesgue's theorem of continuity under the integral sign, $l(A,\cdot)$ is a continuous function, since $f$ and $G$ are continuous in time and bounded.
\end{proof}

Having established the asymptotic behavior of $Y_n$, we focus on continuous-time martingales. In particular, thanks to Lenglart's inequality for continuous-time martingales, we shall estimate, for any $A\in\mathcal{B}_\nu^+$, the functions $l(A,\cdot)$ and $L(A,\cdot)$, where $L(A,\cdot)$ is given by,
\begin{equation}\label{eq:L}
\forall 0\leq t<t^\star(A), ~L(A,t) 	= \int_0^t l(A,s) \ud s .
\end{equation}

		\begin{theo}																			\label{th:mg}
		Let $A\in\mathcal{B}_{\nu}^+$. For each integer $n$, the process $M_{n}( A  , \cdot)$ defined by,
		\begin{equation}\label{Mn}
		\forall 0\leq t<t^\star(A),~ M_{n}( A  , t) = \sum_{i=0}^{n-1} M^{i+1}(t) \mathbf{1}_{\{Z_i\in A\}},
		\end{equation}
		is a continuous-time martingale in the filtration $(\mathcal{G}_n \vee \bigvee_{i=0}^{n-1} \mathcal{F}_t^{i+1})_{0\leq t<t^\star(A)}$.
		\end{theo}
		
\begin{proof}
Let $0\leq s<t<t^\star(A)$. Then,
$$  \esp_{\nu_0}\big[{M}_n(A,t) |{\mathcal{G}}_n \vee \bigvee_{i=0}^{n-1} \mathcal{F}_s^{i+1} \big] =  \sum_{i=0}^{n-1} \esp_{\nu_0}\big[ M^{i+1}(t)\mathbf{1}_{\{Z_i\in A\}} |\mathcal{G}_n \vee \bigvee_{j=0}^{n-1} 			\mathcal{F}_s^{j+1} \big] . $$
Moreover, on the strength of Proposition \ref{prop:indcond}, we have
$$
\bigvee_{j\neq i+1} \mathcal{F}_{s}^j  ~   \underset{\mathcal{G}_n \vee \mathcal{F}_s^{i+1}}{\bot} ~ \mathcal{F}_t^{i+1}
\qquad\text{and}\qquad
\mathcal{F}_t^{i+1} ~ 		\underset{\sigma(Z_i) \vee \mathcal{F}_s^{i+1}}{\bot}			~ \mathcal{G}_n  .
$$
Therefore, since $\sigma(Z_i)$ is a sub-$\sigma$-field of $\mathcal{G}_n$, we have by Corollary 6.8 of \cite{Kal},
$$
\bigvee_{j\neq i+1} \mathcal{F}_{s}^j  ~   \underset{\mathcal{G}_n \vee \mathcal{F}_s^{i+1}}{\bot} ~ \mathcal{F}_t^{i+1}  \vee \sigma(Z_i) 
\qquad\text{and}\qquad
\sigma(Z_i)\vee\mathcal{F}_t^{i+1} ~	\underset{\sigma(Z_i) \vee \mathcal{F}_s^{i+1}}{\bot}	 ~ \mathcal{G}_n .
$$
Thus, as $M^{i+1}(t)\mathbf{1}_{\{Z_i\in A\}}$ is $\sigma(Z_i)\vee\mathcal{F}_t^{i+1}$-measurable,
\begin{eqnarray*}
\esp_{\nu_0}\big[ M_n(A,t) | \mathcal{G}_n \vee \bigvee_{i=0}^{n-1} \mathcal{F}_s^{i+1} \big] &=& 
\sum_{i=0}^{n-1} \esp_{\nu_0}\big[ M^{i+1}(t)\mathbf{1}_{\{Z_i\in A\}} |\mathcal{G}_n \vee \mathcal{F}_s^{i+1} \big]  \\
~ &=&  \sum_{i=0}^{n-1} \esp_{\nu_0}\big[ M^{i+1}(t) \mathbf{1}_{\{Z_i\in A\}} | \sigma(Z_i) \vee \mathcal{F}_s^{i+1} \big] .
\end{eqnarray*}
Furthermore, with Lemma \ref{lem:mg},
$$ \esp_{\nu_0} \big[  {M}^{i+1}(t) \mathbf{1}_{\{Z_i\in A\}} | \sigma(Z_i) \vee \mathcal{F}_s^{i+1} \big] = M^{i+1}(s) \mathbf{1}_{\{Z_i\in A\}}.$$
Thus,
$$ 
\esp_{\nu_0} \big[ {M}_n(A,t) | \mathcal{G}_n \vee \bigvee_{i=0}^{n-1} \mathcal{F}_s^{i+1} \big] =  \sum_{i=0}^{n-1} M^{i+1}(s)\mathbf{1}_{\{Z_i\in A\}} .
$$
By $(\ref{Mn})$, this ensures that $M_n(A,\cdot)$ is a martingale.
\end{proof}

		\noindent
		For any $A\in\mathcal{B}_\nu^+$, let us introduce the counting process $N_n(A,\cdot)$, defined for any $0\leq t<t^\star(A)$ by
		\begin{equation} \label{pms:def:Ncont}
		N_n(A,t) = \sum_{i=0}^{n-1} \mathbf{1}_{\{S_{i+1}\leq t\}} \mathbf{1}_{\{Z_i\in A\}} .
		\end{equation}

		\begin{lem1}																	\label{crochet2}
		Let $A\in\mathcal{B}_\nu^+$. For all integers $n$, the process given for all $0\leq s<t^\star(A)$ by
		\begin{equation} \label{pms:def:Mtilde}
		\widetilde{M}_{n}( A  , s) = \int_0^s Y_{n}( A  , u)^+ \ud {M}_{n}( A , u) ,
		\end{equation}
		is a martingale whose predictable variation process $<\widetilde{M}_{n}(A)>$ satisfies for any $x\in E$,
		$$
		\forall 0\leq s<t^\star(A),~<\widetilde{M}_{n}( A )>(s)  \to 0\quad\!\textit{$\prob_x$-a.s.}~\text{as $n\to+\infty$}.
		$$
		\end{lem1}

\begin{proof}
By $(\ref{Miplus1})$, $(\ref{Mn})$ and $(\ref{pms:def:Ncont})$, for any $0\leq t<t^\star(A)$, one may differently write $M_n(A,t)$,
\begin{equation}\label{pms:expr:Mn}
M_n(A,t) = N_n(A,t) - \int_0^t \sum_{i=0}^{n-1} \mathbf{1}_{\{S_{i+1}\geq u\}} \mathbf{1}_{\{Z_i\in A\}} \lambda(Z_i,u) \ud u .
\end{equation}
In the light of Theorem \ref{th:mg}, this is a continuous-time martingale. As a consequence, the process $\mathcal{A}_n( A ,\cdot)$ given by, 
$$ \forall 0\leq s<t^\star(A), ~\mathcal{A}_n( A ,s) = \int_0^s    \sum_{i=0}^{n-1}  \mathbf{1}_{\{S_{i+1}\geq u\}} \mathbf{1}_{\{Z_i\in A\}} ~\! \lambda(Z_i ,u)   \ud u ,$$
is the compensator of the counting process $N_n(A,\cdot)$. In order to prove that $\widetilde{M}_n(A,\cdot)$ is a martingale, one may only state that
$$ \esp_{\nu_0} \left[ \int_0^t \Big(Y_n(A,s)^+\Big)^2 \ud \mathcal{A}_n( A ,s) \right] < +\infty .$$
Recall that $\lambda$ is bounded on the set $A\times[0,t^\star(A)[$ on the strength of Lemma \ref{majorlambda}. $C$ denotes an upper bound of $\lambda$ on this set. Consequently,
\begin{eqnarray*}
\int_0^t \Big(Y_n( A ,s)^+\Big)^2 \ud \mathcal{A}_n( A ,s) &=& \sum_{i=0}^{n-1} \int_0^t \Big(Y_n( A ,s)^+\Big)^2  \mathbf{1}_{\{S_{i+1}\geq s\}} \mathbf{1}_{\{Z_i\in A\}} \lambda(Z_i,s) \ud s \\
~&\leq& \int_0^t  \Big(Y_n( A ,s)^+\Big)^2 C \sum_{i=0}^{n-1}  \mathbf{1}_{\{S_{i+1}\geq s\}} \mathbf{1}_{\{Z_i\in A\}} \ud s,
\end{eqnarray*}
Furthermore, by $(\ref{eq:Y})$ and $(\ref{invY})$,
\begin{equation}
\int_0^t \Big(Y_n( A ,s)^+\Big)^2 \ud \mathcal{A}_n( A ,s)  \leq  C \int_0^t Y_n( A ,s)^+ \ud s .		\label{major:crochet}
\end{equation}
As $Y_n(A,\cdot)^+$ is bounded by $1$ by $(\ref{ynplus})$, we have
 \begin{eqnarray*}
 \int_0^t \Big(Y_n( A ,s)^+\Big)^2 \ud \mathcal{A}_n( A ,s)  &\leq& C t .
 \end{eqnarray*}
Hence,
$$ \esp_{\nu_0}\left[ \int_0^t \Big(Y_n( A ,s)^+\Big)^2 \ud\mathcal{A}_n( A ,s) \right] < +\infty .$$
This states that $\widetilde{M}_n(A,\cdot)$ is a martingale. The predictable variation of $\widetilde{M}_n(A,\cdot)$ is given by,
$$
\forall 0\leq s<t^\star(A),~<\widetilde{M}_{n}( A )>(s) = 		  \int_0^s \Big(Y_{n}( A ,u)^+\Big)^2 \ud \mathcal{A}_n(A,u)  .
$$
Therefore, by $(\ref{major:crochet})$,
$$\forall 0\leq s <t^\star(A),~<\widetilde{M}_{n}( A )>(s)  \leq C\int_0^s Y_{n}( A ,u)^+ \ud u .$$
According to Lemma \ref{S003}, $Y_n(A,\cdot)^+$ is bounded by $1$ and for any $u$, $Y_n(A,u)^+$ almost surely tends to $0$. Thus, by Lebesgue's dominated convergence theorem,
$$  <\widetilde{M}_{n}( A )>(s)  \longrightarrow 0 ~\textit{$\prob_x$-a.s.}~\text{when $n\to+\infty$}.$$
This achieves the proof.
\end{proof}

\noindent
We stated this lemma in order to apply Lenglart's inequality to the martingale $\widetilde{M}_n(A,\cdot)$. A reference about this inequality for continuous-time martingales may be found in \cite{And}, II.5.2.1. Lenglart's inequality.

\begin{rem}	\label{rem:lenglart}
In the light of Lenglart's inequality, the previous lemma directly induces that for any $A\in\mathcal{B}_\nu^+$, $0\leq t<t^\star(A)$, and $x\in E$,
$$ \sup_{0\leq s\leq t} \Big|\widetilde{M}_n(A,s)\Big| \stackrel{\prob_x~}{\longrightarrow} 0 .$$
\end{rem}

Let $A\in\mathcal{B}_{\nu}^+$. We propose $\widehat{L}_n(A,\cdot)$ as an estimator of the function $L(A,\cdot)$ defined by $(\ref{eq:L})$. It is given by,
\begin{equation}		\label{def:Lchapeau}
\forall 0\leq t<t^\star(A), ~\widehat{L}_{n}(A ,t) 		=	\int_0^t   Y_n(A,s)^+ \ud N_n(A,s) .
\end{equation}
$\widehat{L}_n(A,\cdot)$ is a Nelson-Aalen type estimator of $L(A,\cdot)$. The following results are related to its asymptotic behavior. We shall see that smoothing this estimator provides an estimator of $l(A,\cdot)$. Before, let us introduce this notation, for each $n\geq0$,
\begin{equation} \label{eq:Letoile}
\forall 0\leq t<t^\star(A), ~ {L}_{n}^\ast   (A,t) 	=	\int_0^t l(A , s)  \mathbf{1}_{\{Y_{n} ( A , s) > 0\}} \ud s  .
\end{equation}
We recall that $Y_n(A,\cdot)$ and its generalized inverse have already been defined by $(\ref{eq:Y})$ and $(\ref{invY})$.

		\begin{prop1} 																			\label{pppaaa}
		Let $A\in\mathcal{B}_\nu^+$, $0<t<t^\star(A)$ and $x\in E$. Then,
		$$
		\sup_{0\leq s \leq t}	\Big| \widehat{L}_n(A,s) - L_n^\ast(A,s) \Big| \stackrel{\mathbf{P}_{x}~}{\longrightarrow} 0 ,
		 $$
	 	when $n$ goes to infinity.
		\end{prop1}

\begin{proof}The definition of $\widetilde{M}_n(A,s)$ $(\ref{pms:def:Mtilde})$, the expression of $M_n(A,s)$ $(\ref{pms:expr:Mn})$ and the definition of $\widehat{L}_n(A,s)$ $(\ref{def:Lchapeau})$ yield to
$$
\widetilde{M}_{n}( A  , s)	= \widehat{L}_n(A,s) - \int_0^s Y_n(A,u)^+ \lambda(Z_i,u)  \mathbf{1}_{\{Z_i\in A\}} \mathbf{1}_{\{S_{i+1}\geq u\}} \ud u.
$$
Thus, $\widetilde{M}_{n}( A  , s)$ may be written in the following way.
\begin{equation} \label{hop}
\widetilde{M}_{n}( A  , s) = \widehat{L}_{n}( A  ,s) - L_n^\ast(A,s) - a_n(s),
\end{equation}
where the extra-term $a_n(s)$ is given by
\begin{equation} \label{def:ans}
a_n(s)	= \int_0^s  Y_n(A,u)^+ \sum_{i=0}^{n-1} \big[   \lambda(Z_i ,u) - l( A  ,u) \big] \mathbf{1}_{\{Z_i\in A\}} \mathbf{1}_{\{S_{i+1}\geq u\}}    \ud u  .
\end{equation}
Therefore,
$$ 
\sup_{0\leq s \leq t}	\Big| \widehat{L}_n(A,s) - L_n^\ast(A,s) \Big| \leq \sup_{0\leq s \leq t} \big| \widetilde{M}_n(A,s) \big| + \sup_{0\leq s \leq t} \big| a_n(s)\big| .
$$
We have already stated in Remark \ref{rem:lenglart} that $\displaystyle\sup_{0\leq s \leq t} |\widetilde{M}_n(A,s)|$ tends in probability to $0$. As a consequence, we only need to study the limit of $\displaystyle\sup_{0\leq s \leq t} \big|a_n(s)\big|$. With $(\ref{eq:Y})$, we have
$$
 \big|a_n(s)\big|   \leq	\int_0^t \mathbf{1}_{\{Y_n(A,u)>0\}}  \Bigg| Y_n(A,u)^+ \sum_{i=0}^{n-1} \lambda(Z_i ,u)  \mathbf{1}_{\{Z_i\in A\}} \mathbf{1}_{\{S_{i+1}\geq u\}} -  l(A,u) \Bigg| \ud u .
$$
The integrated function almost surely converges to $0$ on the strength of Proposition \ref{defl}. Furthermore, in the light of Lemma \ref{majorlambda}, there exists a real number $C>0$, which is an upper bound of $\lambda$ on $A\times[0,t^\star(A)[$. Hence, for $u\leq t$,
\begin{eqnarray*}
\Bigg| Y_n(A,u)^+ \sum_{i=0}^{n-1}\lambda(Z_i ,u)  \mathbf{1}_{\{Z_i\in A\}} \mathbf{1}_{\{S_{i+1}\geq u\}} -  l(A,u) \Bigg|  &\leq& \big| l(A,u)\big| + C \\
~&\leq& K_t(A) + C ,
\end{eqnarray*}
where $K_t(A)$ has already been defined by $(\ref{Kt})$. As a consequence, we apply Lebesgue's dominated convergence theorem, and we obtain
$$ \lim_{n\to+\infty} \sup_{0\leq s\leq t} \big| a_n(s)\big|    =   0 \quad \textit{$\prob_{x}$-a.s.,}$$
showing the result.
\end{proof}

		\begin{theo}																		\label{cor-cons}
		Let $A\in\mathcal{B}_\nu^+$, $0<t<t^\star(A)$ and $x\in E$. Then,
		$$ \sup_{0\leq s\leq t} \big| \widehat{L}_{n}(A,s) - L(A,s) \big|  \stackrel{\mathbf{P}_{x}~}{\longrightarrow} 0  ~\text{as $n\to+\infty$}.$$
		\end{theo}
		
\begin{proof}
By the triangle inequality, we have
$$ \big| \widehat{L}_{n}(A,s) - L(A,s) \big| \leq  \big| \widehat{L}_{n}(A,s) - L^\ast_{n} (A,s) \big| +  \big| {L}^\ast_{n}(A,s) - L(A,s) \big| .$$
In the light of Proposition \ref{pppaaa}, we only need to show that $\sup_{0\leq s \leq t} \big| {L}^\ast_{n}(A,s) - L(A,s) \big|$ tends in probability to $0$. By $(\ref{Kt})$, $(\ref{eq:L})$ and $(\ref{eq:Letoile})$,
$$\sup_{0\leq s \leq t} \big| {L}^\ast_{n}(A,s) - L(A,s) \big|        \leq K_t(A) \int_0^t \mathbf{1}_{\{Y_{n}(A,r)  = 0\}} \ud r ,$$
and the bound almost surely converges to $0$ by Lemma \ref{S003}.
\end{proof}

\noindent
The asymptotic normality of the estimator may be stated as follows.

		\begin{prop1} \label{prop:clt}
		Let $A\in\mathcal{B}_\nu^+$, $0<t<t^\star(A)$ and $x\in E$. Assume that the rate of convergence in probability in $(\ref{def:petitl})$ is $o(n^{-1/2})$, that is
		\begin{equation}	\label{def:petitl_sqrtn}
		\sqrt{n}~\! \left( Y_n(A,t)^+~\! \sum_{i=0}^{n-1} \lambda(Z_i,t) \mathbf{1}_{\{Z_i\in A\}} \mathbf{1}_{\{S_{i+1}\geq t\}} - l(A,t) \right) \stackrel{\mathbf{P}_{x}~}{\longrightarrow}  0 ,
		 \end{equation}
		when $n$ goes to infinity. Then, we have the pointwise asymptotic normality. As $n$ goes to infinity,
		\begin{equation*} \label{clt}
		\sqrt{n}\left( \widehat{L}_n(A,t) - L(A,t)\right) \stackrel{\mathcal{D}}{\longrightarrow}~\! \mathcal{N}\left( 0  ,  \int_0^t \frac{l(A,s)}{\int_A G(z,s)\nu(\ud z)} \ud s\right).
		\end{equation*}
		\end{prop1}

\begin{proof}
By $(\ref{hop})$, we have the following equation
$$ \sqrt{n}\left(\widehat{L}_n(A,t) - L(A,t)\right)  =  \sqrt{n} \widetilde{M}_n(A,t) + \sqrt{n} a_n(t) + \sqrt{n}\big( L_n^\ast(A,t) - L(A,t)\big) ,$$
where the continuous-time martingale $\widetilde{M}_n(A,\cdot)$ and the extra-term $a_n$ have already been defined by $(\ref{pms:def:Mtilde})$ and $(\ref{def:ans})$. From Assumption $(\ref{def:petitl_sqrtn})$ and similarly to the proof of Proposition \ref{pppaaa}, we state that $\sqrt{n} a_n(t)$ converges to $0$ in probability. Moreover, by $(\ref{lim1yn})$, $\mathbf{1}_{\{Y_n(A,0)=0\}}$ almost surely tends to $0$. This sequence takes its values on $\{0,1\}$, thus $\sqrt{n}\mathbf{1}_{\{Y_n(A,0)=0\}}$ tends to $0$ too. Since $Y_n(A,\cdot)$ is decreasing, we have
\begin{eqnarray*}
\sqrt{n}\big|L_n^\ast(A,t) - L(A,t)\big| &\leq& \sqrt{n} \int_0^t l(A,s) \mathbf{1}_{\{Y_n(A,s)=0\}}\ud s \\
~&\leq& \sqrt{n}\mathbf{1}_{\{Y_n(A,0)=0\}} \int_0^t l(A,s)\ud s .
\end{eqnarray*}
As a consequence, the term $\sqrt{n}(L_n^\ast(A,t) - L(A,t))$ almost surely converges to $0$. In addition, on the strength of Theorem IV.1.2 of \cite{And}, we have
$$ \sqrt{n}~\!\widetilde{M}_n(A,t) \stackrel{\mathcal{D}}{\longrightarrow}~\! \mathcal{N}\left( 0~\!,  \int_0^t \frac{l(A,s)}{\int_A G(z,s)\nu(\ud z)} \ud s\right) .$$
This states the result.
\end{proof}

\begin{rem}
The hypothesis $(\ref{def:petitl_sqrtn})$ looks like the assumptions A3 or A4 in \cite{MR1345201}, under which the authors state the asymptotic normality of their estimators.
\end{rem}


Let $A\in\mathcal{B}_\nu^+$. We have provided an estimator of the function $L(A,\cdot)$ on the interval $[0,t^\star(A)[$ (see Theorem \ref{cor-cons}). We shall prove that $L(A,\cdot)$ and the cumulative rate $\Lambda(\xi,\cdot)$, defined by $(\ref{pms:def:LAMBDA})$, are close for any $\xi\in A$ if $A$ is small enough. In the same way, we state that $l(A,\cdot)$, given by $(\ref{def:petitl})$, and the jump rate $\lambda(\xi,\cdot)$ are close for any $\xi\in A$.

		\begin{lem1}																				\label{txlip}
		Let $A\in\mathcal{B}_\nu^+$, $z\in A$ and $0\leq s<t^\star(A)$,
			\begin{eqnarray*}
				\big| \lambda(z,s) - l( A ,s) \big| 		&\leq&			 [\lambda]_{Lip} ~ \text{\normalfont diam}~\!  A  ,\\
				 \big| \Lambda(z,s) - L( A ,s) \big|	& \leq&			 s~\! [\lambda]_{Lip} ~ \text{\normalfont diam}~\!  A .
			\end{eqnarray*}
		\end{lem1}

\begin{proof}First, we have
\begin{eqnarray*}
\big| \lambda(z,s) - l( A ,s) \big|
~&=& \left|  \lambda(z,s) - \frac{ \int_{ A } f(\xi,s) \nu(\ud\xi)}{ \int_{ A } G(\xi,s) \nu(\ud\xi)} \right| \\
~&\leq& \frac{1}{ \int_{ A } G(\xi,s) \nu(\ud \xi) }  \int_{ A }  \big|   \lambda(z,s) - \lambda(\xi,s)   \big| G(\xi,s) \nu(\ud \xi)  \\
~&\leq& [\lambda]_{Lip} ~ \text{diam}~\! A .
\end{eqnarray*}
Moreover, integrating the previous result leads to $|\Lambda(z,s)-L( A ,s)|\leq s~\! [\lambda]_{Lip} ~ \text{\normalfont diam}~\!  A$.
\end{proof}

Notice that we can define and estimate the function $L( A ,\cdot)$ only when $\nu( A )$ is strictly positive. Therefore, we need to estimate the indicator function $\mathbf{1}_{\{\nu( A )>0\}}$. For this, we estimate the quantity $\nu( A )$ by its empirical version,
$$\widehat{\nu}_n ( A ) =  \frac{1}{n} \sum_{i=0}^{n-1} \mathbf{1}_{\{Z_i \in  A \}} .$$

		\begin{lem1}
		\label{o01}
		By virtue of the ergodic theorem, $\widehat{\nu}_n( A )\to\nu( A )$ almost surely. Furthermore,
		$$ \mathbf{1}_{\{\widehat{\nu}_n( A ) > n^{-1/2} \}} \to  \mathbf{1}_{\{\nu( A )>0\}}  ~\text{$\prob_{x}$-\textit{a.s.} when $n \to+\infty$} ,$$
		for any $x\in E$.
		\end{lem1}

\begin{proof}Let us distinguish the cases $\nu(A)>0$ and $\nu(A)=0$.
\begin{enumerate}[(i)]
\item If $\nu( A )>0$, $\widehat{\nu}_n( A )$ has an almost sure limit which is strictly positive. Thus, for $n$ large enough, $\widehat{\nu}_n( A ) > n^{-1/2} $ almost surely.
\item If $\nu( A )=0$, the number of visits in $A$ of the Markov chain $(Z_n)_{n\geq0}$ is almost surely finite, because this Markov chain is Harris-recurrent. In this case, the sum $\sum_{i=0}^{n-1} \mathbf{1}_{\{Z_i\in A \}}$ almost surely converges to a finite sum. \qedhere
\end{enumerate}
\end{proof}

\noindent
In the following remark, we focus on the asymptotic behavior of $\displaystyle\sup_{0\leq s\leq t} \widehat{L}_n(A,s)$ when $\nu(A)=0$. Indeed, we stated that the convergence of the estimator holds only when $\nu(A)>0$.

\begin{rem} \label{nuzero}
Let $A\in\mathcal{B}(E)$ such that $\overline{A}\cap\partial E=\emptyset$ and $\nu(A)=0$. Let $0\leq s\leq t<t^\star(A)$ and $x\in E$. By $(\ref{def:Lchapeau})$, we have
$$\widehat{L}_n(A,s) \leq \sum_{i=0}^{n-1} Y_n(A,S_{i+1})^+ \mathbf{1}_{\{Z_i\in A\}}~\prob_x\textit{-a.s.}$$
Together with $(\ref{ynplus})$, we obtain
$$\sup_{0\leq s\leq t}\widehat{L}_n(A,s) \leq \sum_{i=0}^{n-1} \mathbf{1}_{\{Z_i\in A\}} ~\prob_x\textit{-a.s.}$$
In addition, the number of visits in $A$ is almost surely finite since $\nu(A)=0$. Thus, $\displaystyle\sup_{0\leq s\leq t} \widehat{L}_n(A,s)$ almost surely tends to a finite sum.
\end{rem}

For estimating $\Lambda$ on $\mathcal{K}\times [0,t]$ where $\mathcal{K}$ is a compact subset of $E$, one considers a thin enough partition $(A_k)$ of $\mathcal{K}$ and one estimates $\Lambda(\xi,s)$ by $\widehat{L}_n(A_j,s)$, for $\xi\in A_j$.
		
		\begin{theo}
		\label{est-sim-final}
		Let $\mathcal{K}$ be a compact subset of $E$ and $\xi\in E$. For any $\varepsilon, \eta>0$, there exist an integer $N$ and a finite partition $P=(A_k)$ of $\mathcal{K}$, such that for any $n\geq N$, for any $0<t<\min_k t^\star(A_k)$,
		$$
		\prob_{\xi}\Big(  \sup_{x\in\mathcal{K}} \sup_{0\leq s\leq t} \Bigg|      \sum_{k=1}^{|P|}   \widehat{L}_{n}( A _k , s)   \mathbf{1}_{\{  \widehat{\nu}_n( A _k)>\frac{1}{\sqrt{n}}\}} \mathbf{1}_{\{x\in A _k\}}   ~-~\Lambda			(x,s) \mathbf{1}_{\{x \in \mathcal{K}'\}} \Bigg|     > \eta \Big) < \varepsilon ,
		$$
		where $\mathcal{K}'$ is defined by
		$$ \mathcal{K}'  ~=  \bigcup_{\nu( A _k)>0}    A _k .$$
		\end{theo}
		
\begin{proof}Let us fix $s$ and $x$. We have
\begin{align*}
\sum_{k=1}^{|P|} &  \widehat{L}_{n}( A _k , s)   \mathbf{1}_{\{  \widehat{\nu}_n( A _k)>\frac{1}{\sqrt{n}}\}} \mathbf{1}_{\{x\in A _k\}} ~-~\Lambda			(x,s) \mathbf{1}_{\{x \in \mathcal{K}'\}} \\
&=\quad  \sum_{k=1}^{|P|}   \widehat{L}_{n}( A _k , s)   \mathbf{1}_{\{  \widehat{\nu}_n( A _k)>\frac{1}{\sqrt{n}}\}} \mathbf{1}_{\{x\in A _k\}}   ~-~
\sum_{k=1}^{|P|} \Lambda	(x,s) \mathbf{1}_{\{x \in A_k\}} \mathbf{1}_{\{\nu(A_k)>0\}} \\
&= \quad  \sum_{k=1}^{|P|}   \widehat{L}_{n}( A _k , s)   \mathbf{1}_{\{  \widehat{\nu}_n( A _k)>\frac{1}{\sqrt{n}}\}} \mathbf{1}_{\{x\in A _k\}} \mathbf{1}_{\{\nu(A_k)=0\}} \\
&\quad+~ \sum_{k=1}^{|P|} \Big(   \widehat{L}_{n}( A _k , s)   \mathbf{1}_{\{  \widehat{\nu}_n( A _k)>\frac{1}{\sqrt{n}}\}} - \Lambda	(x,s)\Big) \mathbf{1}_{\{x\in A_k\}}  \mathbf{1}_{\{\nu(A_k)>0\}} .
\end{align*}
Let us denote by $l$ the integer between $1$ and $|P|$ such that $x$ is in $A_l$. We distinguish the two cases $\nu(A_l)>0$ and $\nu(A_l)=0$. By assumption on $\mathcal{K}$, $A _l$ is a relatively compact set such that $\overline{A_l}\cap\partial E=\emptyset$.
\begin{enumerate}[(i)]

\item If $\nu(A_l)=0$, $\widehat{L}_n(A_l,s)$ almost surely and uniformly converges to a finite sum, by Remark \ref{nuzero}. Thus, by Lemma \ref{o01},
$$\sup_{0\leq s\leq t}\widehat{L}_n( A _l,s) \mathbf{1}_{\{\widehat{\nu}_n(A _l)>\frac{1}{\sqrt{n}}\}} \stackrel{\prob_x~}{\longrightarrow} 0 .$$

\item If $\nu(A_l)>0$. We have, by the triangle inequality,
\begin{eqnarray*}
\big| \widehat{L}_n( A _l,s) \mathbf{1}_{\{\widehat{\nu}_n( A _l)>\frac{1}{\sqrt{n}}\}} - \Lambda(x,s)  ~ \big| &\leq&\quad
\widehat{L}_n( A _l,s)  \big|   \mathbf{1}_{\{\widehat{\nu}_n( A _l)>n^{-1/2}\}}     -    1 \big|    \\
&~&+~ \big| \widehat{L}_n ( A _l ,s)      -   L( A _l,s)  \big|   \\
&~&+~\Big|   L( A _l,s)     -   \Lambda(x,s) \Big| .
\end{eqnarray*}
\begin{enumerate}
\item The first term uniformly tends to $0$ in probability, since $|\mathbf{1}_{\{\widehat{\nu}_n( A _l)>n^{-1/2}\}}-1|$ almost surely tends to $0$ by Lemma \ref{o01}, and because $\widehat{L}_n( A _l,\cdot)$ uniformly converges in probability to $L( A _l,\cdot)$, according to Theorem \ref{cor-cons}.
\item The second term uniformly tends to $0$ in probability on the strength of Theorem \ref{cor-cons}.
\item The third term is bounded by $t [\lambda]_{Lip} \max_k \text{diam}~\!  A _k$, according to Lemma \ref{txlip}, which does not depend on $x$ and is arbitrarily small.
\end{enumerate}
Since the sum is finite, this achieves the proof. \qedhere
\end{enumerate}
\end{proof}

\begin{rem}
It appears difficult to derive the asymptotic pointwise normality for estimating $\Lambda(\xi,t)$ from the previous results. Indeed, the asymptotic behavior of the difference $\sqrt{n}(\widehat{L}_n(A,t) - \Lambda(\xi,t))$ may be investigated through $\sqrt{n}(\widehat{L}_n(A,t) - L(A,t))$ and $ \sqrt{n}(L(A,t) - \Lambda(\xi,t))$. For the former, we already obtained the asymptotic normality in Proposition \ref{prop:clt}. The main difficulty comes from the later, which goes to plus or minus infinity. A way to overcome this drawback would be to consider a partition depending on $n$. This approach would lead to several technical difficulties and remains an open problem for the authors. Nevertheless, we would like to emphasize that the result presented in Proposition \ref{prop:clt} provides an approximation of $\Lambda(\xi,t)$ with the desired accuracy. In particular, we could obtain from this result an asymptotic confidence interval for $\Lambda(\xi,t)$ with a given confidence level.
\end{rem}

\subsection{Estimation of $\lambda$} In this part, we focus on smoothing the estimator $\widehat{L}_n(A,\cdot)$ of $L(A,\cdot)$ in order to provide a consistent estimator of $l(A,\cdot)$ (see Proposition \ref{kernel}), and, therefore, of $\lambda$ (see Theorem \ref{pms:theo:CVfin}). Indeed, we shall state another corollary of Proposition \ref{pppaaa}, which deals with the smoothing of $\widehat{L}_{n}(A,\cdot)$ by some kernel methods.

Let $K$ be a continuous kernel with support $[-1,1]$. Let us introduce the following notations. For any real number $b>0$ and $0<t<t^\star(A)$, we denote,
\begin{equation}\label{def:hatl}
\forall 0\leq u \leq t, ~ \widehat{l}_{n,b,t} (A,u) = \frac{1}{b} \int_0^{t} K \Big( \frac{u-s}{b} \Big) \ud \widehat{L}_{n}(A,s) .
\end{equation}
This is an estimator of $l(A,\cdot)$ within the interval $[0,t]$. Furthermore, we denote also,
$$\forall 0\leq u\leq t, ~ l_{n,b,t}^\ast(A,u)        =   \frac{1}{b} \int_0^{t} K \Big( \frac{u-s}{b} \Big) \ud {L}_{n}^\ast(A,s) .$$
We first prove the following lemma.

\begin{lem1}	\label{lem:zzz}
Let $A\in\mathcal{B}_\nu^+$, $b>0$ and $0<r<t<t^\star(A)$. 
$$ 
\sup_{0\leq s \leq r} \big| \widehat{l}_{n,b,t} ( A ,s) - l_{n,b,t}^\ast ( A ,s)  \big| \leq \frac{2}{b} V(K) \sup_{0\leq s\leq t} \big| \widehat{L}_{n}( A ,s) - L_{n}^\ast( A ,s) \big| .
$$
\end{lem1}
\begin{proof}
Let us denote,
$$\forall 0\leq s\leq t, ~g(s) = \widehat{L}_{n}( A ,s) - L_{n}^\ast( A ,s) .$$
Thus, by an integration by parts and since $g(0)=0$,
$$ K\big( (s-t)/b \big) g(t) =  \int_0^t K\big( (s-u)/b\big) \ud g(u) - \int_0^{t} g(u-) \ud K\big((s-u)/b\big) .$$
We deduce from the triangle inequality,
$$ b~\! \sup_{0\leq s \leq r}   \big| \widehat{l}_{n,b,t} ( A ,s) - l_{n,b,t}^\ast ( A ,s)  \big| \leq \sup_{0\leq s\leq r} \big| K\big( (s-t)/b \big) g(t) \big|   + \sup_{0\leq s\leq r} \int_0^{t} g(u-) \ud K\big( (s-u)/b \big).$$
For any $v\notin[-1,1]$, $K(v)=0$. This induces that
\begin{eqnarray*}
\big| K\big( (s-t)/b \big) g(t) \big|   	& \leq &    \sup_{0\leq u\leq t} |g(u)|  K\big( (s-t)/b \big) \\
							~ & = &  	\sup_{0\leq u\leq t} |g(u)|  ~\! \Big\{ K\big( (s-t)/b \big)-K(v)\Big\} \\
							~ & \leq &	\sup_{0\leq u\leq t} |g(u)|  ~\! V(K) .
\end{eqnarray*}
Furthermore,
\begin{eqnarray*}
\int_0^{t} g(u-) \ud K\big( (s-u)/b \big) &=& \lim_{\max\{s_{i+1}-s_i\} \to 0} ~  \sum_{i=1}^p g(s_i) ~\! \Big\{ K\big( (s - s_{i+1})/b\big) - K\big( (s-s_i)/b\big)  \Big\} \\
~&\leq& \sup_{0\leq u\leq t} |g(u)| V(K) ,
\end{eqnarray*}
giving the proof.
\end{proof}

		\begin{prop1}															\label{kernel}
		Let $A\in\mathcal{B}_\nu^+$ and $0<r_1<r_2<t<{t^\star( A )}$. There exists a sequence $(\beta_n)_{n\geq0}$, which almost surely tends to $0$, such that
		$$  \sup_{r_1\leq s \leq r_2}   \big| \widehat{l}_{n,\beta_n,t} ( A ,s) - l( A ,s)  \big|   \stackrel{\mathbf{P}_{x}~}{\longrightarrow} 0 ~\text{when $n \to+\infty$} ,$$
		for any $x\in E$.
		\end{prop1}

\begin{proof}
By the triangle inequality, we have for any sequence $(b_n)_{n\geq0}$,
\begin{equation}	\label{eq:itn}
\big| \widehat{l}_{n,b_n,t} ( A ,s) - l( A ,s)  \big|      \leq     \big| \widehat{l}_{n,b_n,t} ( A ,s) - l_{n,b_n,t}^\ast ( A ,s)  \big|    +     \big| {l}^\ast_{n,b_n,t} ( A ,s) - l ( A ,s)  \big| .
\end{equation}
We consider the sequence $(b_n)_{n\geq0}$ defined by,
$$\forall n\geq0,~b_n = \sqrt{ \sup_{0\leq s\leq t} \big| \widehat{L}_{n}(A,s) - L_n^\ast(A,s) \big|} .$$
On the strength of Proposition \ref{pppaaa}, the sequence $(b_n)_{n\geq0}$ tends in probability to $0$, and we have
$$  \sup_{0\leq s\leq t}
	\Big| \widehat{L}_{n}( A ,s)-L_{n}^\ast( A ,s) \Big|
 \stackrel{\prob_{x}}{=} o(b_n) ~\text{when $n \to +\infty$.}$$
Finally, according to Lemma \ref{lem:zzz},
\begin{equation} \label{eq:mmm}
\sup_{r_1\leq s\leq r_2} \big| \widehat{l}_{n,b_n,t} ( A ,s) - l_{n,b_n,t}^\ast ( A ,s)  \big|  \stackrel{\mathbf{P}_{x}~}{\longrightarrow} 0 ~\text{when $n \to +\infty$} .
\end{equation}
Now, we focus on the asymptotic behavior of $\big|{l}^\ast_{n,b_n,t} ( A ,s) - l ( A ,s)\big|$. By the change of variable $s-b_n v = u$, we have
$$ {l}^\ast_{n,b_n,t} ( A ,s) 		=			 \int_{(s-t)/b_n}^{s/b_n} K(v) \mathbf{1}_{\{Y_{n}( A ,s-b_n v)>0\}} l( A ,s-b_n v) \ud v     .$$
As $(b_n)_{n\geq0}$ tends in probability to $0$, there exists a subsequence
$$(\beta_n)_{n\geq0} =(b_{\alpha(n)})_{n\geq0},$$
which almost surely converges to $0$. Let
\begin{equation}	\label{omega1}
\Omega_1 = \big\{ \omega\in\Omega ~:~ \beta_n(\omega) \to 0~\text{and}~\mathbf{1}_{\{Y_n(A,r_2,\omega) = 0\}} \to 0~\!\text{when $n\to+\infty$}\big\} .
\end{equation}
According to foregoing and by $(\ref{lim1yn})$, $\prob_x(\Omega_1) = 1$. Let $\omega\in\Omega_1$ and $N$ such that, for any $n\geq N$,
$$r_1/\beta_n(\omega) \geq 1\qquad\text{and}\qquad (r_2-t)/\beta_n(\omega) \leq -1.$$
For any $n\geq N$, we have
$$
\frac{s}{\beta_n(\omega)}\geq \frac{r_1}{\beta_n(\omega)}\geq 1\qquad\text{and}\qquad \frac{s-t}{\beta_n(\omega)} \leq\frac{r_2-t}{\beta_n(\omega)}\leq-1.
$$
Thus,
$$ {l}^\ast_{n,\beta_n(\omega),t} ( A ,s,\omega)  = \int_{-1}^1 K(v) \mathbf{1}_{\{Y_{n}( A ,s-\beta_n(\omega) v , \omega)>0\}} l( A ,s-\beta_n(\omega) v) \ud v    ,$$
because the support of $K$ is $[-1,1]$. Hence,
\begin{align}
\big|  {l}^\ast_{n,\beta_n(\omega),t} &( A ,s,\omega) - l ( A ,s)  \big|	\nonumber \\
&\leq  \quad  \Big|   \int_{-1}^1 K(u) \Big(  l( A ,s-\beta_n(\omega) u)  -  l( A ,s) \Big) \ud u  \Big|  \nonumber  \\
&\quad+ \Big|  \int_{-1}^1  K(v) \Big(\mathbf{1}_{\{Y_{n}( A ,s-\beta_n(\omega) v , \omega)>0\}}-1\Big) l( A ,s-\beta_n(\omega) v) \ud v \Big|. \label{eq:nnn}
\end{align}
We shall prove that the terms
$$\sup_{r_1\leq s\leq r_2} \Big|   \int_{-1}^1 K(u) \Big(  l( A ,s-\beta_n(\omega) u)  -  l( A ,s) \Big) \ud u  \Big|$$
and
$$ \sup_{r_1\leq s\leq r_2} \Big|  \int_{-1}^1  K(v) \Big(\mathbf{1}_{\{Y_{n}( A ,s-\beta_n(\omega) v , \omega)>0\}}-1\Big) l( A ,s-\beta_n(\omega) v) \ud v \Big|$$
converge to $0$ for any $\omega\in\Omega_1$. We shall begin by the second term.
\begin{enumerate}[(i)]
\item Recall that $K$ and $l(A,\cdot)$ are two continuous functions. Let $C$ defined by
$$C = K_t(A) \sup_{-1\leq v\leq 1} K(v) ,$$
where $K_t(A)$ has already been defined by $(\ref{Kt})$. Then, we have
\begin{align}
\Big|  \int_{-1}^1  K(v) \Big(\mathbf{1}_{\{Y_{n}( A ,s-\beta_n(\omega) v , \omega)>0\}}&-1\Big) l( A ,s-\beta_n(\omega) v) \ud v \Big| \nonumber  \\
& \leq C \int_{-1}^1  \big|   \mathbf{1}_{\{Y_{n}( A ,s-\beta_n(\omega) v , \omega)>0\}}-1   \big|   \ud v  \nonumber \\
& \leq C \int_{-1}^1  \mathbf{1}_{\{Y_{n}( A ,s-\beta_n(\omega) v , \omega) = 0\}}   \ud v . \label{pms:a001}
\end{align}
The change of variable $u=s-\beta_n(\omega) v$ yields to
\begin{eqnarray}
\int_{-1}^1  \mathbf{1}_{\{Y_{n}( A ,s-\beta_n(\omega) v , \omega) = 0\}}   \ud v &\leq&
\frac{1}{\beta_n(\omega)}   \int_{s-\beta_n(\omega)}^{s+\beta_n(\omega)}  \mathbf{1}_{\{Y_n(A,u,\omega)=0\}}  \ud u \nonumber \\
& \leq&  \frac{1}{\beta_n(\omega)}   \int_{r_2-\beta_n(\omega)}^{r_2+\beta_n(\omega)}  \mathbf{1}_{\{Y_n(A,u,\omega)=0\}}  \ud u ,\label{eq:cv001}
\end{eqnarray}
because $u\mapsto\mathbf{1}_{\{Y_n(A,u,\omega)=0\}}$ is an increasing function. Moreover, by definition of Riemann's integral of piecewise-continuous functions, we have when $n\to+\infty$,
$$
\int_{r_2-\beta_n(\omega)}^{r_2+\beta_n(\omega)}  \mathbf{1}_{\{Y_n(A,u,\omega)=0\}}  \ud u  ~\sim~  \beta_n(\omega) \Big(\mathbf{1}_{\{Y_n(A,r_2^-,\omega)=0\}} + \mathbf{1}_{\{Y_n(A,r_2^+,\omega)=0\}}\Big).
$$
Thus, when $n\to+\infty$,
\begin{equation}	\label{eq:cv002}
\frac{1}{\beta_n(\omega)}   \int_{r_2-\beta_n(\omega)}^{r_2+\beta_n(\omega)}  \mathbf{1}_{\{Y_n(A,u,\omega)=0\}}  \ud u  ~ \sim ~ \mathbf{1}_{\{Y_n(A,r_2^-,\omega)=0\}} + \mathbf{1}_{\{Y_n(A,r_2^+,\omega)=0\}} .
\end{equation}
Futhermore, by definition of $\Omega_1$ (see $(\ref{omega1})$),
\begin{equation} \label{eq:cv003}
 \mathbf{1}_{\{Y_n(A,r_2,\omega)=0\}} \to 0\quad\text{when $n\to+\infty$}.
\end{equation}
Finally, $(\ref{pms:a001})$, $(\ref{eq:cv001})$, $(\ref{eq:cv002})$ and $(\ref{eq:cv003})$ show that
$$ \sup_{r_1\leq s\leq r_2} \Big|  \int_{-1}^1  K(v) \Big(\mathbf{1}_{\{Y_{n}( A ,s-\beta_n(\omega) v , \omega)>0\}}-1\Big) l( A ,s-\beta_n(\omega) v) \ud v \Big|$$
tends to $0$ for almost all $\omega$, since $\prob_x(\Omega_1)=1$.
\item By virtue of Lebesgue's dominated convergence theorem, we have for any $\omega\in\Omega_1$,
$$  \sup_{r_1\leq s\leq r_2} \Big|   \int_{-1}^1 K(u) \Big(  l( A ,s-\beta_n(\omega) u)  -  l( A ,s) \Big) \ud u  \Big|   \longrightarrow 0 \quad \text{when $n\to+\infty$} ,$$
since $l(A,\cdot)$ is a continuous function.
\end{enumerate}
Finally, by $(\ref{eq:nnn})$, when $n$ goes to infinity,
$$ \sup_{r_1\leq s\leq r_2} \big| {l}^\ast_{n,\beta_n,t} ( A ,s) - l ( A ,s) \big| \to 0~\prob_{x}\textit{-a.s.}$$
Consequently, the convergence also holds in probability. Together with $(\ref{eq:itn})$ and $(\ref{eq:mmm})$, we obtain the expected result.
\end{proof}

\noindent
Proposition \ref{kernel} states the existence of a (random) sequence $(\beta_n)_{n\geq0}$ such that the estimator $\widehat{l}_{n,\beta_n,t}(A,\cdot)$ tends to $l(A,\cdot)$, within every compact subset $[r_1,r_2]$, with $0<r_1<r_2<t$ and $t<t^\star(A)$. Nevertheless, we do not obtain an explicit construction of this sequence. Furthermore, let us notice that the sequence $(\beta_n)_{n\geq0}$ depends on $t$.

Finally, for any $A\in\mathcal{B}_\nu^+$, we provide a consistent estimator $\widehat{l}_{n,\beta_n,t}( A ,\cdot)$ of $l( A ,\cdot)$ within the interval $[0,t]$ for any $0<t<t^\star(A)$. However, we are interested in the estimation of the jump rate $\lambda$. Since $\lambda$ is Lipschitz, if $A$ is small enough, then $l( A ,\cdot)$ and $\lambda(\xi,\cdot)$ are close for $\xi\in A$ (see Lemma \ref{txlip}). This leads to our main result of convergence.

		\begin{theo}
		\label{pms:theo:CVfin}
		Let $\mathcal{K}$ be a compact subset of $E$ and $\xi\in E$. For any $\varepsilon, \eta > 0$, there exist an integer $N$ and a finite partition $P=(A _k)$ of $\mathcal{K}$ such that, for any $0<t<\min_k t^\star(A_k)$, there exists for each $k$, a sequence $(\beta_n(A_k))_{n\geq0}$ (depending on $t$) which almost surely tends to $0$, such that for any $n\geq N$, for any $0<r_1<r_2<t$,
		$$\prob_{\xi}\Big(  \sup_{x\in\mathcal{K}} \sup_{r_1\leq s\leq r_2} \Bigg|      \sum_{k=1}^{|P|}   \widehat{l}_{n,\beta_n(A_k),t}( A _k , s)   \mathbf{1}_{\{\widehat{\nu}_n( A _k)>\frac{1}{\sqrt{n}}\}} \mathbf{1}_{\{x\in A _k\}}   ~-~\lambda(x,s) \mathbf{1}_{\{x \in \mathcal{K}'\}} \Bigg| > \eta \Big) < \varepsilon ,$$
		where $\mathcal{K}'$ is defined by
		$$ \mathcal{K}'  ~=  \bigcup_{\nu( A _k)>0}    A _k .$$
		\end{theo}

\begin{proof}
This is a consequence of Proposition \ref{kernel} and Lemma \ref{txlip}. The proof is similar to the one of Theorem \ref{est-sim-final}.
\end{proof}

\begin{rem}	\label{pms:rem:K}
If the compact subset $\mathcal{K}$ is close to $E$, then, for any partition $(A_k)$ de $\mathcal{K}$, the lower bound $t^\star(A_k)$ is small. In this case, one estimates the jump rate on a great part of the state space, but within a small time interval. Conversely, if one chooses a small compact subset $\mathcal{K}$, centered in $E$, one estimates the jump rate on a small part of the state space, but within a larger time interval.
\end{rem}

		\section{Illustration}
		\label{s:simu}


This section is devoted to a simulation study, in order to illustrate the convergence results stated in Theorems \ref{est-sim-final} and \ref{pms:theo:CVfin}. We shall present a numerical example which can be directly connected to the pratical application given in Section \ref{s:def}.

We consider here a non-homogeneous marked renewal process $(X_t)_{t\geq0}$ defined on the interval $E=]0,60[$ and starting from the point $30$. At time $t$, $X_t$ models the temperature configuration in degree Celsius of a production machine. The temperature is piecewise-constant over time. When a failure occurs, the maintainer instantaneously repairs the machine and changes the temperature configuration: he tries to obtain the normal regime at $20^\circ$C. However, he does a random error whose variance depends on the difference between the previous temperature and $20^\circ$C. We are interested in the estimation of the failure rate in the normal temperature configuration. The three characteristics $\lambda$, $Q$, and $t^\star$ of $(X_t)_{t\geq0}$ are given for any $x\in E$ by
\begin{itemize}
\item for any $A\in\calB(E)$,
$$Q(x,A) = \frac{1}{K_x}\int_A\mathbf{1}_E(y) \exp\left(-\frac{(y-20)^2}{2\sigma_x^2}\right),$$
with $\sigma_x=0.5+|x-20|$,
\item for any $t\geq0$, $\lambda(x,t) = 3+0.05 x$,
\item $t^\star(x)=1$,
\end{itemize}
where $K_x$ is the normalizing constant. This transition kernel obviously satisfies Doeblin's condition. Thus, Assumption \ref{hyp:ergodic} holds for the embedded chain $(Z_n)_{n\geq0}$. Moreover, the cumulative failure rate $\Lambda$ associated to $\lambda$ is given by $\Lambda(x,t) = (3+0.05x)\,t$.

We simulate long trajectories of the process: the observation of $200$, $300$ or $400$ jumps is available for estimating the failure rate $\lambda$ and its integrate version $\Lambda$. As mentioned before, we focus on the estimation of $\Lambda(x,t)$ and $\lambda(x,t)$ for $x=20$. We choose to approximate $\lambda(x,t)$ by the function $l(A,t)$, with $A=\{y\in E\,:\,|y-x|\leq2\}$. As $t^\star(A)=1$, we estimate $l(A,t)$ by $(\ref{def:hatl})$ which is an integral between times $0$ and $0.9$. Thus, the convergence result stated in Theorem \ref{pms:theo:CVfin} is valid on every compact subset in $]0,0.9[$. We present our simulation study over the interval $[0,0.8]$ for the estimation of $\Lambda$ and the compact subset $[0.2,0.8]\subset]0,0.9[$ for $\lambda$.

For each simulated trajectory, the number of visits in $A$ is close to $73.5\%$ of the number of observed jumps. In addition, the chosen bandwidth $\beta_n(A)$ can be written in the following way,
$$\beta_n(A)=\frac{1}{h_n(A)^\alpha} ,$$
where $h_n(A)$ denotes the (random) number of visits in $A$, and $\alpha=1/4$. Figure \ref{lab:fig1} presents the estimation of both the functions of interest from the observation of $400$ jumps. In addition, we focus on the evaluation of the error for different sample sizes: $200$, $300$ and $400$ observed jumps. For each sample size, we generate $100$ replicates of the model. We provide in Figure \ref{lab:fig2} the boxplots of the integrated square error between the functions of interest and their estimates. The integrated square error is not surprisingly decreasing with the sample size. Both these figures illustrate the good behavior of our estimation method.

We measure computational times for both simulating and estimating with our 1.86GHz Intel Core 2 Duo processor. The simulation of one trajectory of $400$ jumps needs only $0.143$ seconds. Our estimation procedure from $400$ jumps provides both the estimators over the intervals $[0,0.8]$ and $[0.2,0.8]$ in $1.399$ seconds.

	\begin{figure}[!h]
	\centering
	\includegraphics[width=0.48\textwidth]{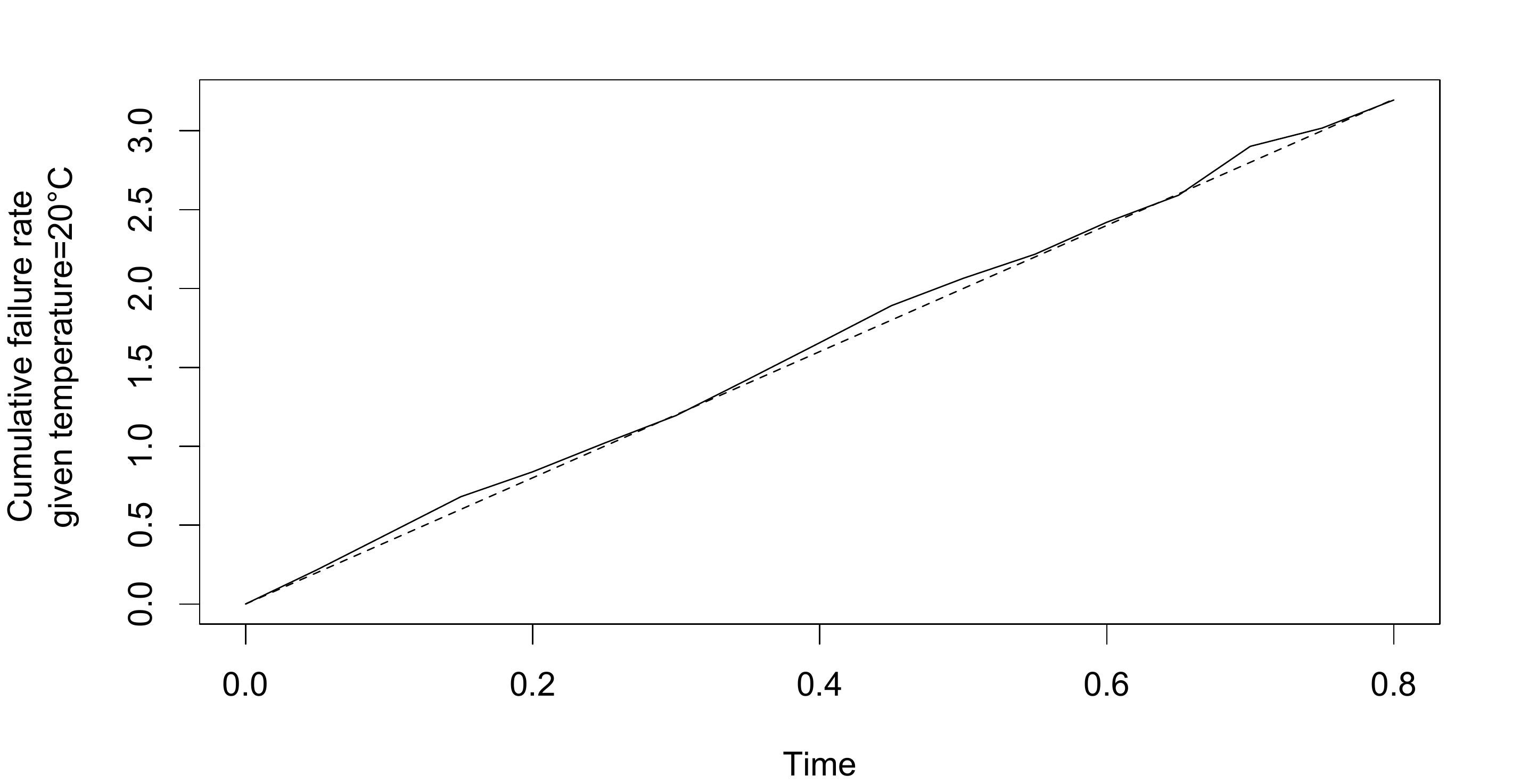}
	\includegraphics[width=0.48\textwidth]{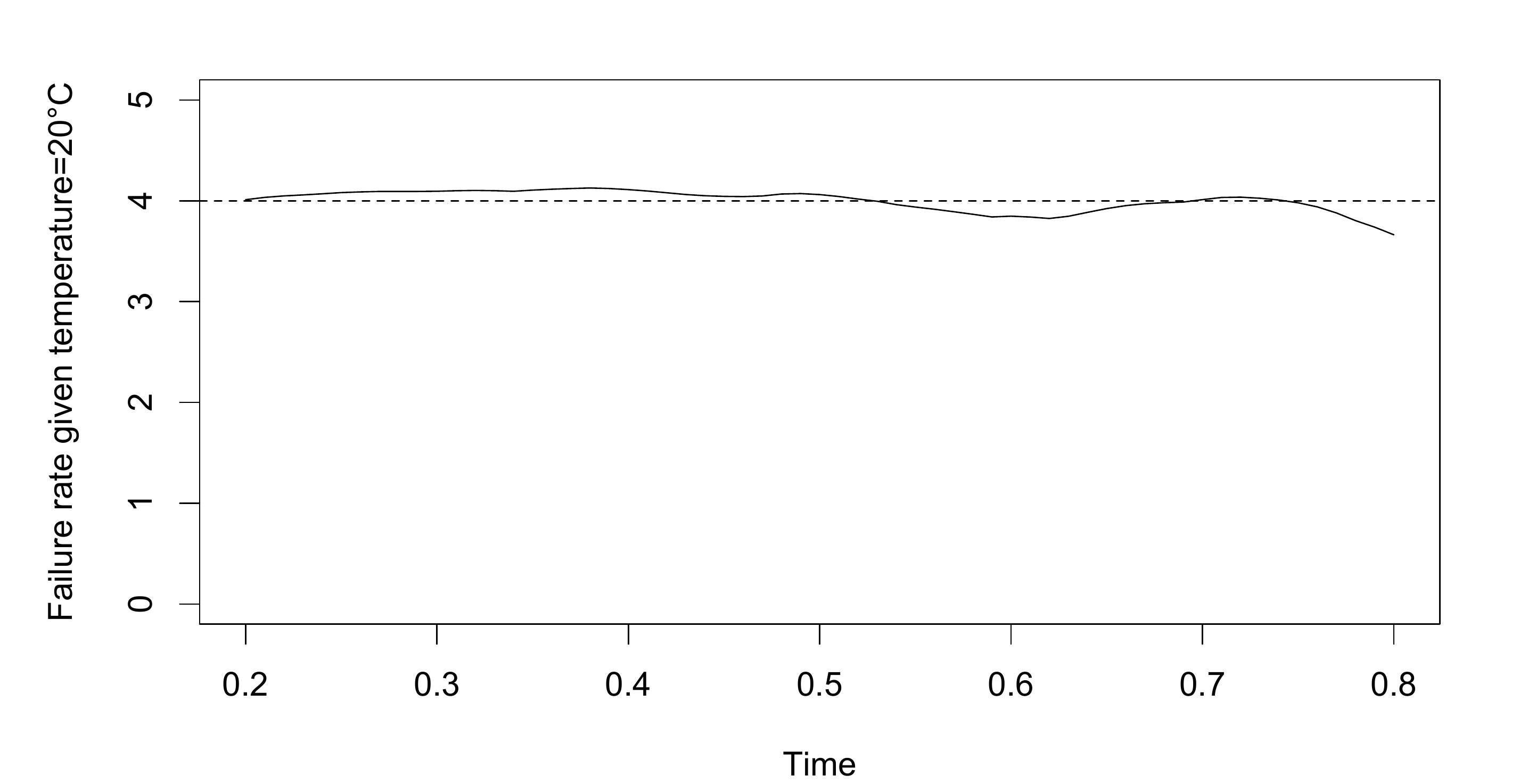}
	\caption{Estimation of the cumulative failure rate (left) and failure rate (right) given the temperature is $20^\circ$C from the observation of $400$ jumps. Estimates are drawn in solid lines, exact rates are in dashed lines.}
	\label{lab:fig1}
	\end{figure}
	
	\begin{figure}[!h]
	\includegraphics[width=0.48\textwidth]{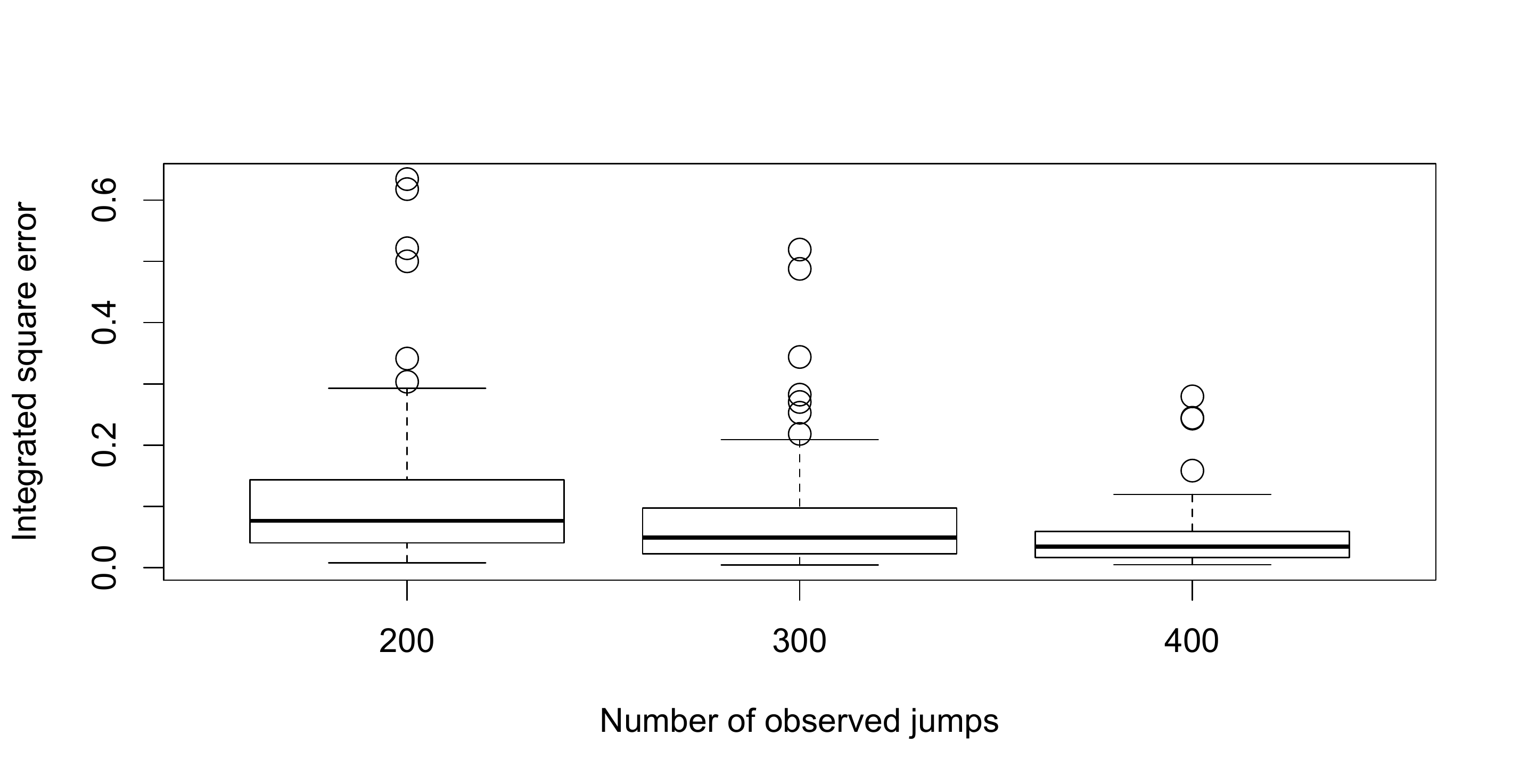}
	\includegraphics[width=0.48\textwidth]{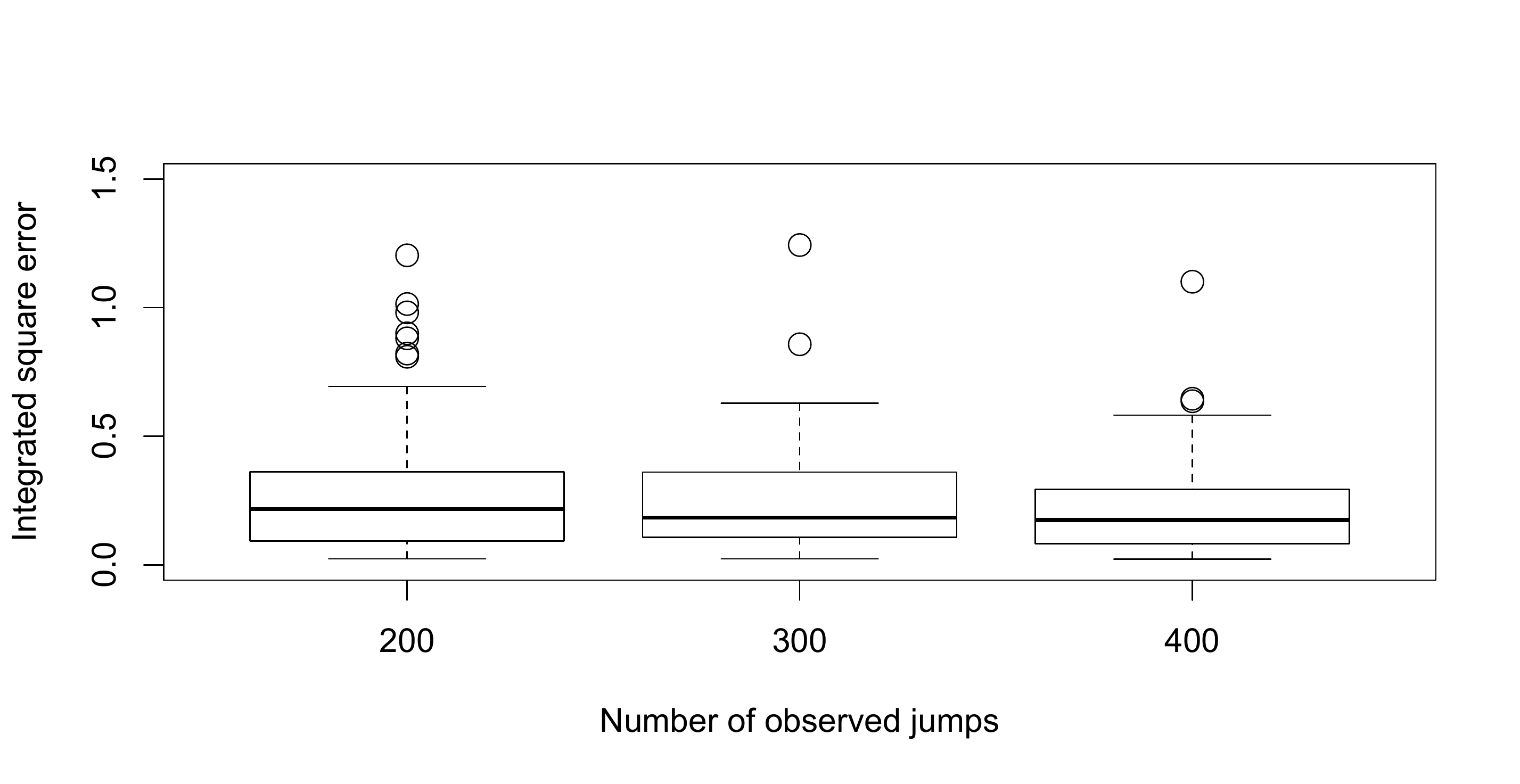}
	\caption{Boxplots of the integrated square error in the estimation of the cumulative failure rate (left) and failure rate (right) given the temperature is $20^\circ$C.}
	\label{lab:fig2}
	\end{figure}

		\section{Concluding remarks}
		
In the present paper, we have developed a powerful nonparametric method for estimating both the cumulative rate and the jump rate for a general class of non-homogeneous marked renewal processes. Two inherent difficulties are related to the presence of a deterministic censorship, and the absence of conditions on the state space of the process. Furthermore, the proposed estimation procedure needs only one observation of the process within a long time interval. In addition, the assumptions which we impose are directly connected to the primitive characteristics of the process. 
		
In this context, we proposed nonparametric estimators of the cumulative rate and the jump rate, and we proved results of uniform convergence in probability on every compact subset. As it is illustrated in Section \ref{s:simu}, the numerical behavior of both estimators is satisfactory on finite sample size. Furthermore, the method is easy to implement and not time-consuming. Finally, although this paper has an intrinsic interest, it is also a keystone for estimating the conditional density of the interarrival times for piecewise-deterministic Markov processes (see \cite{Az}).

	\appendix
	\section{Proofs of Section \ref{s:def}} \label{app1}
	
		\subsection{Proof of Proposition \ref{prop:indcond}}
		
We state the first conditional independence. Let us consider $h_1,\dots,h_n$ some bounded measurable functions mapping from $\R_+$ to $\R$. From $(\ref{dyn:S})$, we deduce
\begin{align}
\mathbf{E}_{\nu_0} \big[ h_1(S_1) & \dots h_n(S_n) | \mathcal{G}_n\big]\!  \nonumber\\
&= \mathbf{E}_{\nu_0}\Big[ h_1(S_1)\dots h_{n-1}(S_{n-1}) \mathbf{E}_{\nu_0} \big[ h_n(S_n) | \mathcal{G}_n \vee \sigma(S_1,\dots,S_{n-1})\big] \big| \mathcal{G}_n \Big] \nonumber\\
&= \mathbf{E}_{\nu_0}\Big[ h_1(S_1)\dots h_{n-1}(S_{n-1}) \mathbf{E}_{\nu_0} \big[ h_n(S_n) | \mathcal{G}_n \big] \big| \mathcal{G}_n \Big] \nonumber \\
&= \mathbf{E}_{\nu_0}\big[ h_1(S_1)\dots h_{n-1}(S_{n-1}) | \mathcal{G}_n\big]  \mathbf{E}_{\nu_0} \big[ h_n(S_n) | \mathcal{G}_n \big] . \label{ind1}
\end{align}
Moreover, by $(\ref{dyn:Z})$,
\begin{eqnarray*}
\mathbf{E}_{\nu_0} \big[ h_1(S_1) \dots h_{n-1}(S_{n-1}) | \mathcal{G}_n\big]\! &=&\!\mathbf{E}_{\nu_0}\Big[~\! \mathbf{E}_{\nu_0} \big[ h_1(S_1) \dots h_{n-1}(S_{n-1}) | \mathcal{G}_{n-1}\!\vee\! \sigma(\varepsilon_{n-2})\big] \big| \!\mathcal{G}_n\!\Big]\! . 
\end{eqnarray*}
The product $h_1(S_1)\dots h_{n-1}(S_{n-1})$ is $\mathcal{G}_{n-1}\vee\sigma(\delta_0,\dots,\delta_{n-2})$-measurable and
$$ \sigma(\varepsilon_{n-2}) ~\bot ~ \mathcal{G}_{n-1}~\! \vee~\! \sigma(\delta_0 , \dots, \delta_{n-2}) .$$
Together with $(3)$ \cite[page 308]{Chung}, we obtain
$$  \mathbf{E}_{\nu_0} \big[ h_1(S_1) \dots h_n(S_{n-1}) | \mathcal{G}_{n-1} \vee \sigma(\varepsilon_{n-2})\big]   =  \mathbf{E}_{\nu_0} \big[ h_1(S_1) \dots h_{n-1}(S_{n-1}) | \mathcal{G}_{n-1} \big] .$$
Finally, we have
\begin{equation*} \mathbf{E}_{\nu_0} \big[ h_1(S_1) \dots h_{n-1}(S_{n-1}) | \mathcal{G}_n\big]  = \mathbf{E}_{\nu_0} \big[ h_1(S_1) \dots h_{n-1}(S_{n-1}) | \mathcal{G}_{n-1} \big]  .\end{equation*}
In the light of $(\ref{ind1})$, we obtain
$$ \mathbf{E}_{\nu_0} \big[ h_1(S_1) \dots h_n(S_n) | \mathcal{G}_n\big]  =   \mathbf{E}_{\nu_0} \big[ h_1(S_1) \dots h_{n-1}(S_{n-1}) | \mathcal{G}_{n-1}\big] \mathbf{E}_{\nu_0} \big[ h_n(S_n) | \mathcal{G}_n\big] .$$
Thus, a straightforward induction leads to
$$ \mathbf{E}_{\nu_0} \big[ h_1(S_1) \dots h_n(S_n) | \mathcal{G}_n\big]  = \prod_{i=1}^n \esp_{\nu_0} \big[h_i(S_i) | \mathcal{G}_{i} \big] .$$
Furthermore, taking for $j\neq i$, $h_j = \mathbf{1}$, yields to
$$ \mathbf{E}_{\nu_0} \big[ h_i(S_i) | \mathcal{G}_n\big] = \mathbf{E}_{\nu_0} \big[ h_i(S_i) | \mathcal{G}_i\big] .$$
Hence,
$$ \mathbf{E}_{\nu_0} \big[ h_1(S_1) \dots h_n(S_n) | \mathcal{G}_n\big]  =  \prod_{i=1}^n \esp_{\nu_0} \big[h_i(S_i) | \mathcal{G}_n \big] .$$
Thus, we have
$$ \bigvee_{j\neq i} \sigma(S_j) ~   \underset{\mathcal{G}_n}{\bot} ~  \sigma(S_i) ,$$
that immediately implies the expected result. The second conditional independence is straightforward from $(\ref{dyn:S})$. \fin

		\subsection{Proof of Lemma \ref{lem:mg}}
		
Let $0\leq s<t<t^\star(Z_i)$. In order to prove that $M^{i+1}$ is a martingale, we have to show that
\begin{align}
&\mathbf{E}_{\nu_0} \left[ N^{i+1}(t) - \int_0^t \lambda(Z_i,u) \mathbf{1}_{\{S_{i+1}\geq u\}} \ud u | \sigma(Z_i)\vee\mathcal{F}_s^{i+1}\right] = M^{i+1}(s)			\nonumber \\
\Leftrightarrow~& \mathbf{E}_{\nu_0}\left[ \mathbf{1}_{\{s<S_{i+1}\leq t\}} -\int_s^t \lambda(Z_i,u)\mathbf{1}_{\{S_{i+1}\geq u\}} \ud u | \sigma(Z_i)\vee \mathcal{F}_s^{i+1}\right] = 0		\nonumber\\
\Leftrightarrow~& \mathbf{1}_{\{S_{i+1}>s\}} \mathbf{E}_{\nu_0}\left[\mathbf{1}_{\{S_{i+1}\leq t\}} | \sigma(Z_i)\vee\mathcal{F}_s^{i+1}\right]  \nonumber \\
&\quad = \mathbf{1}_{\{S_{i+1}>s\}} \mathbf{E}_{\nu_0}\left[ \int_s^t \lambda(Z_i,u)\mathbf{1}_{\{S_{i+1}\geq u\}} \ud u | \sigma(Z_i)\vee \mathcal{F}_s^{i+1}\right] . \label{MGcont}
\end{align}
First, we shall prove that $(\ref{MGcont})$ is equivalent to
\begin{align}
\mathbf{1}_{\{S_{i+1}>s\}}& \esp_{\nu_0}\big[\mathbf{1}_{\{S_{i+1}\leq t\}} | \sigma(Z_i)\vee\sigma(N^{i+1}(s))\big] \nonumber  \\
&=  \mathbf{1}_{\{S_{i+1}>s\}} \mathbf{E}_{\nu_0}\left[ \int_s^t \lambda(Z_i,u)\mathbf{1}_{\{S_{i+1}\geq u\}} \ud u | \sigma(Z_i)\vee\sigma(N^{i+1}(s))\right] . \label{MGcont2}
\end{align}
In the light of Lemma 6.2 of \cite{Kal}, if the following conditions are satisfied,
$$ \{S_{i+1}>s\} \in \Big(\sigma(Z_i) \vee \mathcal{F}_s^{i+1}\Big) \cap \Big(\sigma(Z_i) \vee \sigma(N^{i+1}(s))\Big)$$
and
$$ \{S_{i+1}>s\} \cap \Big(\sigma(Z_i) \vee \mathcal{F}_s^{i+1}\Big)  = \{S_{i+1}>s\} \cap \Big( \sigma(Z_i) \vee \sigma(N^{i+1}(s))\Big),$$
then $(\ref{MGcont})$ and $(\ref{MGcont2})$ are equivalent. The first condition obviously holds. On the other hand, as $\sigma(N^{i+1}(s))$ is a sub-$\sigma$-field of $\mathcal{F}_s^{i+1}$, we have
$$ \{S_{i+1}>s\} \cap \Big(\sigma(Z_i) \vee \mathcal{F}_s^{i+1}\Big)  \supset \{S_{i+1}>s\} \cap \Big(\sigma(Z_i) \vee \sigma(N^{i+1}(s))\Big) .$$
Hence, we only have to prove the reciprocal inclusion,
$$ \{S_{i+1}>s\} \cap \Big( \sigma(Z_i) \vee \mathcal{F}_s^{i+1}\Big)  \subset  \{S_{i+1}>s\} \cap \Big(\sigma(Z_i) \vee \sigma(N^{i+1}(s))\Big) .$$
Let us consider the $\lambda$-system (see for instance Definition 1.10 of \cite{KLENKE}) $\mathcal{D}$ given by
$$
\mathcal{D}=\left\{ C \in \sigma(Z_i)\vee\mathcal{F}_s^{i+1} ~:~\{S_{i+1}>s\}\cap C\in\{S_{i+1}>s\}\cap \Big(\sigma(Z_i)\vee\sigma(N^{i+1}(s))\Big) \right\} ,
$$
and the $\pi$-system (see Definition 1.1 of \cite{KLENKE}) $\mathcal{C}$ defined by
$$
\mathcal{C} = \big\{ \Omega , \{S_{i+1}\leq u\}\cap\{Z_i\in B\}, \{Z_i\in B\} ~:~ 0\leq u\leq s, B\in\mathcal{B}(E) \big\} .
$$
We immediately obtain that $\mathcal{C}\subset\mathcal{D}$ and $\sigma(\mathcal{C})=\sigma(Z_i)\vee\mathcal{F}_s^{i+1}$. Moreover, $\sigma(Z_i)\vee\mathcal{F}_s^{i+1}\subset\mathcal{D}$ on the strength of the monotone class theorem (see for example Theorem 1.19 of \cite{KLENKE}). From the definition of $\mathcal{D}$, we straightforward deduce the reciprocal inclusion. Consequently, $(\ref{MGcont})$ and $(\ref{MGcont2})$ are equivalent. Now, we only have to verify that $(\ref{MGcont2})$ holds. On the one hand, we have
\begin{align*}
\mathbf{1}_{\{S_{i+1}>s\}}& \esp_{\nu_0}\big[\mathbf{1}_{\{S_{i+1}\leq t\}} | \sigma(Z_i)\vee\sigma(N^{i+1}(s))\big] \\
&= \mathbf{1}_{\{S_{i+1}>s\}} \prob_{\nu_0}\big( S_{i+1}\leq t | Z_i , S_{i+1}>s\big) \\
&=  \mathbf{1}_{\{S_{i+1}>s\}} \frac{ \prob_{\nu_0}\big(s< S_{i+1}\leq t | Z_i \big)}{\prob_{\nu_0}\big( S_{i+1}> s| Z_i\big)} .
\end{align*}
From the definitions of $f$ $(\ref{pms:def:f})$ and $G$ $(\ref{pms:def:G})$, we deduce that
\begin{align*}
\mathbf{1}_{\{S_{i+1}>s\}}& \esp_{\nu_0}\big[\mathbf{1}_{\{S_{i+1}\leq t\}} | \sigma(Z_i)\vee\sigma(N^{i+1}(s))\big] \\
&= \mathbf{1}_{\{S_{i+1}>s\}} \frac{\int_s^t f(Z_i,u)\ud u}{G(Z_i,s)} .
\end{align*}
On the other hand,
\begin{align*}
\mathbf{1}_{\{S_{i+1}>s\}} &\esp_{\nu_0}\left[ \int_s^t \lambda(Z_i,u) \mathbf{1}_{\{S_{i+1}\geq u\}}\ud u | \sigma(Z_i)\vee\sigma(N^{i+1}(s))\right] \\
&= \mathbf{1}_{\{S_{i+1}>s\}} \int_s^t \lambda(Z_i,u) \prob_{\nu_0}\big(S_{i+1}\geq u | Z_i , S_{i+1}>s\big) \ud u \\
&=\mathbf{1}_{\{S_{i+1}>s\}} \int_s^t \lambda(Z_i ,u) \frac{ \prob_{\nu_0}\big(S_{i+1}\geq u | Z_i\big)}{\prob_{\nu_0}\big(S_{i+1}>s | Z_i\big)}\ud u .
\end{align*}
Thus, since $u\leq t<t^\star(Z_i)$,
\begin{align*}
\mathbf{1}_{\{S_{i+1}>s\}} &\esp_{\nu_0}\left[ \int_s^t \lambda(Z_i,u) \mathbf{1}_{\{S_{i+1}\geq u\}}\ud u | \sigma(Z_i)\vee\sigma(N^{i+1}(s))\right] \\
&=\mathbf{1}_{\{S_{i+1}>s\}} \frac{\int_s^t \lambda(Z_i ,u)G(Z_i,u) \ud u}{G(Z_i,s)} \\
&= \mathbf{1}_{\{S_{i+1}>s\}} \frac{\int_s^t f(Z_i,u)\ud u}{G(Z_i,s)}.
\end{align*}
As a consequence, we proved $(\ref{MGcont2})$ and the expected result. \fin

		\section{Proofs of Section \ref{s:continuous}} \label{app2}
	
			\subsection{Proof of Proposition \ref{prop:ergodic}}
			~
\begin{enumerate}
\item Let $x\in E$ and $\nu_0=\delta_{\{x\}}$. Then, from Assumption \ref{hyp:ergodic},
$$ \lim_{n\to+\infty} \| Q^n(x,\cdot) - \nu\|_{TV} = 0 .$$
Let us consider $A\in\mathcal{B}(E)$ such that $\nu(A)>0$. Thus, $Q^n(x,A)>0$ for $n$ large enough. In the light of Proposition 4.2.1 of \cite{MandT}, the Markov chain $(Z_n)_{n\geq0}$ is, therefore, $\nu$-irreducible.
\item Furthermore, on the strength of Theorem 4.3.3 of \cite{HLL}, $(Z_n)_{n\geq0}$ is positive Harris-recurrent and aperiodic.
\item A positive Harris-recurrent Markov chain admits a unique (up to a multiple constant) invariant measure (see for instance the introduction of \cite{MR1712909}). \fin
\end{enumerate}

			\subsection{Proof of Lemma \ref{lem37}}
			
Let $g$ be a measurable function bounded by $1$. According to Lemma \ref{desint},
\begin{eqnarray*}
\left| \int_{{E}\times\mathbf{R}_+} g(z,s) \big[\eta_n(\ud z \times \ud s) - \eta(\ud z \times \ud s)\big] \right| &\leq &  \|h\|_{\infty} \| \nu_n - \nu \|_{TV} ,
\end{eqnarray*}
on the strength of Fubini's theorem, where the function $h$ is defined by,
$$ \forall z\in{E},~h(z) = \int_{\mathbf{R}_+} g(z,s) \mu_z(\ud s) .$$
Furthermore, since $|h(z)|\leq 1$ for any $z$, we have
\begin{eqnarray*}
\left| \int_{{E}\times\mathbf{R}_+} g(z,s) \big(\eta_n(\ud z \times \ud s) - \eta(\ud z \times \ud s)\big) \right| 	&\leq &	\|\nu_n - \nu\|_{TV},
\end{eqnarray*}
Thus, by Assumption \ref{hyp:ergodic}, $\| \eta_n - \eta \|_{TV}$ tends to $0$. \fin

			\subsection{Proof of Lemma \ref{r001}}

By assumption, $t^\star$ is continuous. Moreover, for any $\xi\notin\partial E$, $t^\star(\xi)>0$. Since $\overline{A}$ is a compact set such that $\overline{A}\cap\partial E=\emptyset$, we have
$$ \inf_{\xi\in\overline{A}} t^\star(\xi) > 0 .$$
Hence,
$$ \inf_{\xi\in A} t^\star(\xi) \geq \inf_{\xi\in\overline{A}} t^\star(\xi) > 0 .$$
For any $s$, the function $\lambda(\cdot,s)$ is continuous because it is Lipschitz. Furthermore, $\lambda(\xi,\cdot)$ is bounded by $M$, not depending on $\xi$, and locally integrable. Recall that we have
$$ G(\xi, t) = \exp\left(-\int_0^t \lambda(\xi,s)\ud s \right).$$
Thus, by Lebesgue's theorem of continuity under the integral sign, $G(\cdot,t)$ is continuous, and satisfies $G(\xi,t)>0$ for any $\xi$. Hence,
$$\inf_{\xi\in A} G(\xi,t) \geq \inf_{\xi\in\overline{A}} G(\xi,t) >0 .$$
Both inequalities are, therefore, proved. \fin


\nocite{*}
\bibliographystyle{acm}
\bibliography{nhmrp_main_v2} 


\end{document}